\documentclass[a4paper]{article}

\usepackage[utf8]{inputenc}
\usepackage[T1]{fontenc}
\usepackage{amsmath,mathtools,amsfonts,amsthm}
\usepackage{siunitx}
\usepackage{algpseudocode,algorithm}
\usepackage{xcolor}
\usepackage{hyperref}

\newcommand\D{\mathrm{d}}
\newcommand\N{\mathbb{N}}
\newcommand\R{\mathbb{R}}

\newtheorem{theorem}{Theorem}[section]
\newtheorem{lemma}{Lemma}[section]
\newtheorem{corollary}{Corollary}[section]
\newtheorem{remark}{Remark}[section]
\newtheorem{proposition}{Proposition}[section]
\newtheorem{assumption}{Assumption}

\title{Optimal optical conditions for Microalgal production in photobioreactors}

\author{Olivier Bernard$^{1,2}$, Liu-Di Lu$^{1,3,4}$}
\date{}

\begin{document}

\maketitle
{%
\noindent{\small\textit{$^1$INRIA Sophia Antipolis M\'editerran\'ee, BIOCORE Project-Team, Universit\'e Nice C\^ote d'Azur, 2004, Route des Lucioles - BP 93, 06902 Sophia-Antipolis Cedex, France}}\\
{\small\textit{$^2$Sorbonne Universi\'e, INSU-CNRS, Laboratoire d’Oc\'eanographie de Villefranche, 181 Chemin du Lazaret, 06230 Villefranche-sur-mer, France}}\\
{\small\textit{$^3$INRIA Paris, ANGE Project-Team, 75589 Paris Cedex 12, France}}\\
{\small\textit{$^4$Sorbonne Universit\'e, CNRS, Laboratoire Jacques-Louis Lions, 75005 Paris, France}}\\%
}

\begin{abstract}
The potential of industrial applications for microalgae has motivated their recent fast development.   
Their growth dynamics depends on different factors that must be optimized. 
Since they get their energy from photosynthesis, light is a key factor that strongly influences  their productivity. 
Light is absorbed and scattered in the liquid medium, and irradiance exponentially decreases towards the darkest part of the photobioreactor at a rate nonlinearly depending on the biomass concentration. 
Maximizing productivity is then a tricky problem, especially when the growth rate is inhibited by an excess of light. Productivity optimization turns out to be highly dependent on how light is distributed along the reactor, and is therefore related to the extinction rate and the background turbidity.
The concept of optical depth productivity is introduced for systems where background turbidity must be accounted for and a global optimum maximizing productivity is proposed, extending the concept of the compensation condition.
This optimal condition consists in compensating the algal growth rate at the bottom of the reactor by the respiration.
This condition can drive the optimization of the surface biomass productivity depending on the minimum reachable depth.
We develop a nonlinear controller and prove the global asymptotic stability of the biomass concentration towards the desired optimal value.
\end{abstract}


\section{Introduction}

Microalgae are photosynthetic microorganisms whose potential has been highlighted in the last decades, especially for renewable energy and wastewater treatment~\cite{Wijffels2010,Hu2008,Scott2010}.
Compared with terrestrial plants, whose growth is slower due to CO$_2$ availability, the high actual photosynthetic yield of microalgae cultures leads to higher biomass production potential.
Some algal species can be grown to target numerous high added value commercial applications: pharmaceutical, cosmetics or food industries~\cite{Eppink2017,Spolaore2006}.
These microorganisms are generally cultivated at industrial scale in simpler and cheaper open raceway ponds or in high-tech closed photobioreactors, using solar or artificial light depending on the applications.

The biomass productivity in these reactors depends on the photosynthesis efficiency resulting from the light distribution along depth.
Light intensity is strongly attenuated in the photobioreactor due to the absorption and scattering of the microalgae and the background turbidity of the cultivation medium.
Depending on the position of the algal cells in the reactor, they perceive different light intensities which further influence the photon harvesting dynamics~\cite{Demory2018}.
Light attenuation is generally  described by a Beer-Lambert law where the light extinction rate varies with the process type and algal concentration.
Some studies have more accurately represented the way light is attenuated in the process, especially to deal with very dense multi-scattering medium where a photon can be scattered several times before being eventually absorbed~\cite{Morel1988}.
The influence of the background turbidity and the reactor depth on the average growth rate of the algae have also been considered in~\cite{Martinez201811}.

There exists models of various complexities accounting for different mechanisms modulating the growth of microalgae. 
They were used to define conditions optimizing the microalgal productivity, especially by determining an optimal algal biomass so that light is adequately distributed in the reactor~\cite{Grognard2014,Bernard202103,DelaHozSiegler2012,Cornet2009,Cornet2010}.
For instance, authors in~\cite{Masci2010} have used the Droop model~\cite{Droop1968,Droop1983} combining with the growth rate description in~\cite{Bernard2009} to give an optimal condition to maximize the surface biomass productivity.
This so-called \textit{compensation condition} consists in cancelling the net growth rate at the reactor bottom. 
It is used in practice to adapt the biomass concentration to the reactor depth. 
This condition has been numerically shown to be still valid for more complicated cases where both light and temperature vary in~\cite{Bernard202011} in a raceway pond with a varying volume.

On top of light extinction and water depth which have been already deeply studied, in this paper, we focus on understanding how the productivity in photobioreactors is also influenced by the background turbidity. 

Our first contribution, was to extend the work in~\cite{Masci2010} by choosing a more realistic description of the algal growth dealing with photoinhibition.
Our second contribution consisted in  considering a general biomass dependent light extinction function accounting for the background turbidity of the system.
The concept of optical depth productivity is introduced and a condition is derived on the optical depth for globally maximizing productivity.
This optimum corresponds to the \textit{compensation condition}.
We then use this optimal condition to characterize the optimization of the surface biomass productivity depending on the minimum achievable water depth. 
When the light extinction rate is affine with respect to the algal biomass, an upper limit to the productivity (obtained for an infinitely small depth) is given.
A nonlinear controller is given and is proved to stabilize the evolution of the biomass towards the optimal desired value.
The optimal behaviours are illustrated in different cases by numerical experiments.

This paper is organised as follows.
In Section~\ref{sec:model}, we define the key concepts such as average growth rate and light distribution.
We then study the optimization problem in Section~\ref{sec:opt}. 
More precisely we investigate the global behaviour of the optical depth productivity and the optimal condition in Subsection~\ref{subsec:globopt}.
The optimal biomass concentration for a given reactor depth to maximize the surface biomass productivity is investigated in Subsection~\ref{subsec:locopt}.
A nonlinear controller is then introduced in Section~\ref{sec:controller} to stabilize the biomass concentration towards its optimal value.
We illustrate and discuss the behaviour of the optima in different cases by some numerical experiments in Section~\ref{sec:num}.

\section{Description of the model}\label{sec:model}

For a given light intensity $I$ [$\si{\mu mol.m^{-2}.s^{-1}}$], the growth rate of microalgae is defined by a Haldane-type description parametrized as in~\cite{Bernard2012} 
\begin{equation}\label{eq:haldane}
\mu(I) :=  \mu_{\max} \frac{I}{I+\frac{\mu_{\max}}{\theta}(\frac{I}{I^*}-1)^2},
\end{equation}
where $\theta$ is the initial slope of $\mu$ [$d^{-1}$], $\mu_{\max}$ denotes the maximum value of $\mu$ and $I^*$ represents the optimal light intensity. 
This description results from a mechanical consideration of the light harvesting dynamics represented by the Han system~\cite{Han2002} (see~\ref{app:relation} for more details).
The light attenuation is described by a Beer-Lambert law
\begin{equation}\label{eq:beer}
I(X, z) := I_s\exp\Big( \varepsilon(X) z \Big),
\end{equation}
where $X$ [$\si{g.m^{-3}}$] represents the biomass concentration, $z\in[-h,0]$ denotes the vertical position of the algal cells with $h$ [$\si{m}$] the depth and $I_s$ is the light intensity at the reactor surface.
The light extinction $\varepsilon$, which summarises the light absorption and diffusion, is considered to be correlated to the biomass concentration $X$
\begin{equation}\label{eq:eps}
\varepsilon(X) := \alpha_0(s) X^s + \alpha_1,
\end{equation}
where $0<s\leq1$, $\alpha_0(s)>0$ [$\si{m^2.g^{-1}}$] stands for the specific light extinction coefficient of the microalgae species. 
It depends on the parameter $s$.  
The background turbidity, $\alpha_1$ [$\si{m^{-1}}$], is  due to all non-microalgae components i.e. suspended solids and dissolved colored material.
The dependence of $s$ in $\alpha_0$ will be omitted hereafter when no confusion may occur.

From~\eqref{eq:beer} one can compute the mean light intensity received by the algae culture $\bar I = I_s\int_{-h}^0 e^{\varepsilon(X) z} \D z = \frac{I_s}{\varepsilon(X)} \Big( 1 - e^{-\varepsilon(X) h} \Big)$.
This quantity is a decreasing function of $\varepsilon(X)$, which confirms the intuition that a higher biomass concentration or a higher background turbidity leads to lower mean light received in the reactor, due to stronger light attenuation.

Replacing $I$ by~\eqref{eq:beer} in~\eqref{eq:haldane}, one can see that the growth rate varies with depth of the reactor.
Lower growth rate in the upper part of the reactor results from the photo-inhibition caused by the high light intensity.
Similarly, growth rate is weak in the lower part of the reactor because of the low light intensity.
The mean growth rate in the reactor is defined by
\begin{equation}\label{eq:mubar}
\bar \mu(X,h) := \frac 1h \int_{-h}^0 \mu(I(X, z)) \D z.
\end{equation}
Applying then a change of the variable $y = \varepsilon(X) z$, it can be written as
\begin{equation}\label{eq:mubary}
\bar \mu(X,h) = \frac 1{\varepsilon(X)h} \int_0^{\varepsilon(X)h} \mu(I(-y)) \D y,
\end{equation}
so that the mean growth rate depends on \textit{the optical depth} $\varepsilon(X)h$.
This quantity is denoted by $Y$ [-] hereafter.
In this case, the average growth rate~\eqref{eq:mubary} can also be written as a function of $Y$ (i.e. $\bar \mu(X,h)=\bar \mu(Y)$).
Our aim is to optimize the surface biomass productivity (units: $\si{g \cdot m^{-2} \cdot d^{-1}}$) which is defined by
\begin{equation}\label{eq:Pi}
\Pi := (\bar \mu - R) X h.
\end{equation}

\begin{remark}
The evolution of the biomass concentration $X$ is given by
\begin{equation}\label{eq:dotX}
\dot X = (\bar \mu - R - D) X,
\end{equation}
where $R$ [$\si{d^{-1}}$] is the respiration rate and $D$ [$\si{d^{-1}}$] denotes the reactor dilution rate. 
Note that at equilibrium, the biomass surface productivity $\Pi$ is the product between dilution rate ($D = \bar \mu - R$) and surface biomass $X h$.

Note that a nonlinear controller for $D$ is introduced in section~\ref{sec:controller} to stabilize~\eqref{eq:dotX} to the value of $X$ optimizing productivity. 
\end{remark}

\section{Analysis of the optimal productivity}\label{sec:opt}

In this section, we investigate the optimization problem associated with the productivity $\Pi$.
Note that the biomass concentration $X$ and the depth $h$ are both defined on $\R_+$.

\subsection{{Optical depth productivity}}\label{subsec:globopt}

First of all, let us define the optical depth productivity (units: \si{d^{-1}}) by
\begin{equation}\label{eq:P}
P:=(\bar \mu - R) Y.
\end{equation}

\begin{remark}
According to the definition of the optical depth productivity~\eqref{eq:P}, a thin reactor with high biomass concentration is equivalent to a deep reactor with low biomass concentration as long as they both share the same optical depth $Y$.
A low value of $Y$ means a weaker photon harvesting since less light is absorbed.
On the reverse, a too high $Y$ means that light hardly reaches  the bottom of the reactor, with an area where respiration (loss of CO$_2$) exceeds growth (fixation of  CO$_2$). 
Hence, it is necessary to determine the value of $Y$ which maximizes the efficiency of the productivity $P$.
\end{remark}

\begin{theorem}\label{thm:Popt}
Given a surface light intensity $I_s$, there exists an optimal value $Y_{\text{opt}}$ which maximizes the optical productivity $P$.
This value satisfies $\mu \left( I(Y_{\text{opt}}) \right) = R$ and can be expressed explicitly according to the growth rate at the surface $\mu(I_s)$:
\begin{equation}\label{eq:Yopt}
Y_{\text{opt}} =
\left\{
\begin{array}{lr}
\ln \left(\frac{\frac{2I_s R\mu_{\max}}{\alpha {I^*}^2}}{\mu_{\max} - R + \frac{2R\mu_{\max}}{\alpha I^*} - \sqrt{(\mu_{\max} - R)(\mu_{\max} - R + \frac{4R\mu_{\max}}{\alpha I^*})}}\right), \quad \mu(I_s) > R,\\
\ln \left(\frac{\frac{2I_s R\mu_{\max}}{\alpha {I^*}^2}}{\mu_{\max} - R + \frac{2R\mu_{\max}}{\alpha I^*} + \sqrt{(\mu_{\max} - R)(\mu_{\max} - R + \frac{4R\mu_{\max}}{\alpha I^*})}}\right), \quad \mu(I_s) \leq R.
\end{array}
\right.
\end{equation}
\end{theorem}

\begin{proof}
For a given $Y$, the optical productivity $P$ can be written from~\eqref{eq:mubary} and~\eqref{eq:P}
\begin{align}
P(Y) = &\int_0^{Y_{\text{opt}}} \mu(I(-y)) - R \D y + \int_{Y_{\text{opt}}}^Y \mu(I(-y)) - R \D y \nonumber\\
= &P(Y_{\text{opt}}) + \int_{Y_{\text{opt}}}^Y \mu(I(-y)) - R \D y, \label{eq:second_line}
\end{align}
where $Y_{\text{opt}}$ is chosen according to~\eqref{eq:Yopt}. 
On the other hand, the function 
\begin{equation*}
\mu(I(-y)) = \frac{\mu_{\max}I(-y)}{I(-y)+\frac{\mu_{\max}}{\theta}(\frac{I(-y)}{I^*}-1)^2},
\end{equation*}
is concave with respect to $y>0$, so that there exists a $y^*$ such that $\mu(I(-y^*)) = \mu_{\max}$.
According to the value of $\mu(I(0))$ (i.e. $\mu(I_s)$), two cases must be considered:
\begin{itemize}
\item if $\mu(I(0)) = \mu(I_s) >R$, the second term of~\eqref{eq:second_line} is always negative. 
Indeed, in the case where $Y$ is smaller than $Y_{\text{opt}}$, using the concavity of $\mu(I(-y))$, one finds $\mu\left( I(-y) \right) > R$, $\forall y < Y_{\text{opt}}$.
In other words, the second term of~\eqref{eq:second_line} removes the microalgae which grow more than they respire.
Otherwise, one finds  $\mu\left( I(-y) \right) < R$, $\forall y > Y_{\text{opt}}$ (for the same reason as above).
This means that the second term of~\eqref{eq:second_line} adds the microalgae which respire more than their growth.
\item if $\mu(I(0)) = \mu(I_s) \leq R$, then there exists a value $\tilde y \in [0,y^*)$ such that $\mu(I(-\tilde y)) = R$. 
Then if $Y$ is greater than $\tilde y$, the second term of~\eqref{eq:second_line} is negative for the same reason as above.
Otherwise, the productivity $P(Y)$ is negative.
\end{itemize}
In both cases, the second term of~\eqref{eq:second_line} is negative. 
Thus $Y_{\text{opt}}$ maximizes the quantity $P$.

In order to compute $Y_{\text{opt}}$, one needs to solve $\mu ( I ) = R$, or equivalently: 
\begin{equation*}
\frac{R\mu_{\max}}{\theta {I^*}^2}I^2 +(R - \mu_{\max} - \frac{2R\mu_{\max}}{\theta I^*})I + \frac{R\mu_{\max}}{\theta}  = 0.
\end{equation*}
The discriminant of this second order polynomial equation is given by $\Delta = (\mu_{\max} - R)(\mu_{\max} - R + \frac{4R\mu_{\max}}{\theta I^*}) > 0$, which implies that there exists two real roots.
The sum and the product of two roots are both positive, hence both of these two roots are also positive.
Finally $Y_{\text{opt}}$ can be determined by the growth rate at the surface $\mu(I_s)$:
\begin{itemize}
\item if $\mu(I_s) > R$, then there exists one root in the interval $(0, I_s)$ and one root in the interval $(I_s,+\infty)$.
In this case, one has
\begin{equation*}
Y_{\text{opt}} = \ln \left(\frac{\frac{2I_s R\mu_{\max}}{\theta {I^*}^2}}{\mu_{\max} - R + \frac{2R\mu_{\max}}{\theta I^*} - \sqrt{(\mu_{\max} - R)(\mu_{\max} - R + \frac{4R\mu_{\max}}{\theta I^*})}}\right).
\end{equation*}
\item if $\mu(I_s) \leq R$, then two roots both lies into the interval $(0, I_s]$.
In this case, we choose the smaller one (since it represents the light at lower part of the reactors)
\begin{equation*}
Y_{\text{opt}} = \ln \left(\frac{\frac{2I_s R\mu_{\max}}{\theta {I^*}^2}}{\mu_{\max} - R + \frac{2R\mu_{\max}}{\theta I^*} + \sqrt{(\mu_{\max} - R)(\mu_{\max} - R + \frac{4R\mu_{\max}}{\theta I^*})}}\right).
\end{equation*}
\end{itemize}
This concludes the proof.
\end{proof}

\begin{remark}\label{rem:Yopt}
As shown in~\eqref{eq:Yopt}, the value of $Y_{\text{opt}}$ only depends on the model parameters ($\theta$, $\mu_{\max}$, $I^*$, $R$) and on the light intensity at the reactor surface $I_s$. 
In other words, the cancellation of the net growth rate at the bottom of the reactor is the optimal strategy to maximize optical depth productivity (see in Figure~\ref{fig:Pmuy} for illustrations).
\end{remark}

\subsection{Surface biomass productivity}\label{subsec:locopt}

In this section, we focus on the surface biomass productivity $\Pi$.
From the definition of $\Pi$~\eqref{eq:Pi} and the definition of $P$~\eqref{eq:P}, one has
\begin{equation}\label{eq:PtoPi}
\Pi = \frac{X}{\varepsilon(X)}P.
\end{equation}
In general, it is not possible to apply the same strategy (as in the proof of Theorem~\ref{thm:Popt}) to optimize $\Pi$, since $P$ and $\Pi$ do generally not have the same behaviour. 
Only in the case where $s=1$ and $\alpha_1=0$, the factor $\frac{X}{\varepsilon(X)}$ simplifies, leading to the same optimum.
Using then Theorem~\ref{thm:Popt}, we deduce directly the following results.

\begin{corollary}\label{cor:yoptnotur}
If the light extinction function defined by~\eqref{eq:eps} satisfies $\alpha_1 = 0$ and $s=1$, then $Y_{\text{opt}}$ defined by~\eqref{eq:Yopt} maximizes the productivity $\Pi$ and $Y_{\text{opt}}$ is the global optimum.
Moreover, $\tilde Y_{\text{opt}} := Y_{\text{opt}} / \alpha_0$ is the optimal surface biomass.
\end{corollary}

\begin{proof}
Since $\alpha_1 = 0$ and $s=1$, $Y = \varepsilon(X) h = \alpha_0 \tilde Y$ with $\tilde Y := X h$ the surface biomass.
Meanwhile, using~\eqref{eq:PtoPi}, one has $P(\cdot) = \alpha_0\Pi(\cdot)$, then following the same analysis, one finds that $Y_{\text{opt}}$ maximizes $P(\cdot)$, therefore the productivity $\Pi(\cdot)$.
Finally, $\tilde Y_{\text{opt}}$ is given by $Y_{\text{opt}} / \alpha_0$.
\end{proof}

\begin{corollary}\label{cor:hopt}
If the objective is to reach a biomass concentration $X_1$, there exists a unique reactor depth $h_1$ which satisfies $\varepsilon(X_1)h_1 = Y_{\text{opt}}$ and maximizes the productivity $\Pi(X_1,\cdot)$ for this biomass target.
\end{corollary}

\begin{proof}
Since $X_1$ is fixed, then using~\eqref{eq:PtoPi}, one has directly that the optimum is given by $Y_{\text{opt}}$.
In this case, $h_1$ is defined by $Y_{\text{opt}}/\varepsilon(X_1)$.
\end{proof}

In Corollary~\ref{cor:hopt}, we have studied the case with a fixed biomass concentration $X$.
This result does not depend on the considered law $\varepsilon(X)$.
However, optimizing $X$ is more tricky, Corollary~\ref{cor:yoptnotur} provides a result in the case with a specific value of $\alpha_1$ and $s$.
In more general case, the strategy used in the proof of Theorem~\ref{thm:Popt} may fail when optimizing $X$.
We first consider the general case where the background turbidity $\alpha_1$ is not zero.
We split the study into two parts, depending on the value of $s$.

\subsubsection{The standard case (\texorpdfstring{$s = 1$}{Lg})}

Let us start with a technical lemma.

\begin{lemma}\label{lem:dirX}
Let $h_1$ a given depth and $(X_1, h_1)$ which satisfies $Y_{\text{opt}}=\varepsilon(X_1)h_1$. 
Then one has: 
\begin{equation*}
\forall X \in (X_1,\frac{\alpha_1}{\alpha_0 R Y_{\text{opt}}}P(Y_{\text{opt}}) ),\quad  \Pi(X, h_1) > \Pi(X_1, h_1).
\end{equation*}
\end{lemma}

\begin{proof}
Let $X \in (X_1,\frac{\alpha_1}{\alpha_0 R Y_{\text{opt}}}P(Y_{\text{opt}}))$ and denote $Y = \varepsilon(X)h_1$.
Note that $Y > Y_{\text{opt}}$.
On the other hand, one has
\begin{equation*}
\begin{split}
&\Pi(X, h_1)  - \Pi(X_1, h_1) \\
= &\frac{X}{\varepsilon(X)} P(Y) - \frac{X_1}{\varepsilon(X_1)} P(Y_{\text{opt}}) \\
= &\frac{X}{\varepsilon(X)} \Big(P(Y) - P(Y_{\text{opt}})\Big) + \Big(\frac{X}{\varepsilon(X)} - \frac{X_1}{\varepsilon(X_1)}\Big) P(Y_{\text{opt}}) \\
= & \frac{X}{\varepsilon(X)} \int_{Y_{\text{opt}}}^Y \mu(I(-y)) -R\D y + \frac{\alpha_1(X-X_1)}{\varepsilon(X)\varepsilon(X_1)}P(Y_{\text{opt}}) \\
> & -\frac{X}{\varepsilon(X)}R(Y - Y_{\text{opt}}) + \frac{\alpha_1(X-X_1)}{\varepsilon(X)\varepsilon(X_1)}P(Y_{\text{opt}}) \\
= & -\frac{X}{\varepsilon(X)}R\alpha_0(X - X_1) h_1 + \frac{\alpha_1(X-X_1)}{\varepsilon(X)\varepsilon(X_1)}P(Y_{\text{opt}}) \\
> & -\frac{\alpha_1P(Y_{\text{opt}}) R\alpha_0}{\varepsilon(X)\alpha_0 R Y_{\text{opt}}}(X - X_1) h_1 + \frac{\alpha_1(X-X_1)}{\varepsilon(X)\varepsilon(X_1)}P(Y_{\text{opt}}) \\
= & -\frac{\alpha_1(X - X_1)}{\varepsilon(X)\varepsilon(X_1)}P(Y_{\text{opt}}) + \frac{\alpha_1(X-X_1)}{\varepsilon(X)\varepsilon(X_1)}P(Y_{\text{opt}}) \\
= & 0.
\end{split}
\end{equation*}
where we use the fact that $\mu(0) = 0$ (see~\eqref{eq:haldane}), thus $\mu(I(-y))>0$, $\forall y\in[Y_{\text{opt}},Y]$.
Indeed, since $Y =\varepsilon(X)h_1$ and $Y>Y_{\text{opt}}$, using~\eqref{eq:beer} one has $0 < I_s\exp(-Y) < I_s\exp(-Y_{\text{opt}}).$
Finally, since $\mu$ is concave with respect to $I>0$, one finds $\mu(I(-Y)) = \mu(I_s\exp(-Y)) >0$.
This concludes the proof.
\end{proof}

According to Corollary~\ref{cor:hopt}, the couple $(X_1 , h_1)$ satisfies $\varepsilon(X_1)h_1=Y_{\text{opt}}$ and corresponds to the optimum of $\Pi(X_1,\cdot)$ for a given $X_1$.
However the previous lemma implies that this is not the optimal condition to optimize $\Pi(\cdot,h_1)$ for a given $h_1$.
This lemma then enables us to prove the next theorem.

\begin{theorem}\label{thm:globopt}
If $\alpha_1>0$, there is no global optimum for the productivity $\Pi(\cdot,\cdot)$ in $\R_+ \times \R_+$.
\end{theorem}

\begin{proof}
Let us assume that there exists a global optimum for the productivity $\Pi$ denoted by $(X^*,h^*)$.
Since $(X^*,h^*)$ is a global optimum, in particular, this is an optimum in the direction of $h$.
Using Corollary~\ref{cor:hopt}, we find $\varepsilon(X^*)h^* = Y_{\text{opt}}$.
However, using Lemma~\ref{lem:dirX}, there exists $\tilde X^*>X^*$ such that $\Pi(\tilde X^*,h^*)>\Pi(X^*,h^*)$.
This contradicts the fact that $(X^*,h^*)$ is a global optimum.
Therefore, the productivity $\Pi(\cdot,\cdot)$ has no global optimum.
\end{proof}

Since no global optimum for the productivity $\Pi$ can be found when $\alpha_1>0$, then we would like to study the asymptotic behaviour of $\Pi$.
In the following, we focus on the optimum in the direction of $X$ and in the direction of $h$ separately.
Given $X_0$, consider the sequence $(X_n ,h_n)_{n\in\N}$ defined by
\begin{equation}\label{eq:optseq}
h_n = \frac{Y_{\text{opt}}}{\varepsilon(X_{n-1})}, \quad X_n :={\rm argmax}_{X\in\R_+} \Pi(X,h_n).
\end{equation}
It follows from this definition that the sequence $(X_{n-1} ,h_n)_{n>0}$ corresponds to the optimum in the direction of $h$ for $X_{n-1}$, whereas the sequence $(X_n ,h_n)_{n>0}$ corresponds to the optimum in the direction of $X$ for $h_n$.
In other words, these two sequences defined by~\eqref{eq:optseq} aim at searching the local optima by optimizing in the direction of $h$ and in the direction of $X$ alternately. 
A first property of the sequence $(X_n ,h_n)_{n>0}$ is given in next lemma.

\begin{lemma}\label{lem:limprod}
The sequence $(X_n, h_n)$ verifies $\lim_{n\rightarrow \infty} \varepsilon(X_n)h_n = Y_{\text{opt}}$, and the growth rate at the reactor bottom satisfies $\lim_{n\rightarrow \infty}  \mu(I(X_n,h_n)) = R$.
\end{lemma}

\begin{proof}
From the definition of the sequence $(X_{n-1} ,h_n)_{n>0}$, one has $\varepsilon(X_{n-1})h_n = Y_{\text{opt}}$, which means 
\begin{equation*}
\lim_{n\rightarrow \infty}\varepsilon(X_n)h_n = Y_{\text{opt}}\lim_{n\rightarrow \infty} \frac{\varepsilon(X_{n-1})}{\varepsilon(X_n)}  = Y_{\text{opt}}.
\end{equation*}
Denoting by $Y_n= \varepsilon(X_n)h_n$, since $\mu\circ I(-y)$ is a continuous function with respect to $y$ in $\R_+$, one has 
\begin{equation*}
\lim_{n\rightarrow \infty}  \mu(I(X_n, h_n))  = \lim_{n\rightarrow \infty}  \mu(I(Y_n)) = \mu(I(Y_{\text{opt}})) = R. 
\end{equation*}
This concludes the proof.
\end{proof}

This lemma enables us to prove the next theorem, without constraint on the minimal reactor depth.

\begin{theorem}\label{thm:lim}
$\lim_{n\rightarrow \infty} X_n = \infty$, $\lim_{n\rightarrow \infty} h_n = 0$ and $\lim_{n\rightarrow \infty} \Pi(X_n, h_n) = \frac{P(Y_{\text{opt}})}{\alpha_0}$.
\end{theorem}

\begin{proof}
By Lemma~\ref{lem:dirX}, one has $(X_n)_{n\in \N}$ which is a strictly increasing sequence. 
Hence the sequence $(h_n)_{n\in\N^*}$ is strictly decreasing by its construction~\eqref{eq:optseq}.
Since for each $n\in\N^*$, $h_n>0$, then this sequence converges to a limit that we denote by $h_{\lim}$.
Assume that $h_{\lim}>0$, then from~\eqref{eq:optseq}, one has 
\begin{equation*}
h_{\lim} = \lim_{n\rightarrow \infty}h_n = \lim_{n\rightarrow \infty}\frac{Y_{\text{opt}}}{\varepsilon(X_{n-1})}  = Y_{\text{opt}}\lim_{n\rightarrow \infty}\frac1{\alpha_0X_{n-1}+\alpha_1},
\end{equation*}
which means that $\lim_{n\rightarrow \infty} X_n =: X_{\lim} < \infty$.
Then $(X_{\lim} , h_{\lim} )$ is a global optimum, hence a contradiction with Theorem~\ref{thm:globopt}.
Therefore $h_{\lim}=0$, which means that $X_{\lim} = \infty$.

On the other hand, by the construction of these two sequences $(X_{n-1} ,h_n)_{n>0}$, $(X_n ,h_n)_{n>0}$ and Lemma~\ref{lem:dirX}, one has 
\begin{equation*}
\Pi(X_{n-1}, h_n)< \Pi(X_n, h_n) < \Pi(X_n, h_{n+1}).
\end{equation*}
Using~\eqref{eq:PtoPi}, one has $\Pi(X_{n-1}, h_n) = \frac{X_{n-1}}{\varepsilon(X_{n-1})}P(Y_{\text{opt}})$ and $\Pi(X_n, h_{n+1}) = \frac{X_n}{\varepsilon(X_n)}P(Y_{\text{opt}})$.
Taking the limit similarly as in the above inequalities gives $\lim_{n\rightarrow \infty} \Pi(X_n, h_n) = \frac{P(Y_{\text{opt}})}{\alpha_0}$.
This concludes the proof.
\end{proof}

\subsubsection{The case \texorpdfstring{$s < 1$}{Lg}}

In the case $s < 1$, it is much more complicated to follow the same strategy as in the proof of Lemma~\ref{lem:dirX} to show that the couple $(X_1,h_1)$ satisfying $\varepsilon(X_1)h_1=Y_{\text{opt}}$ is not the optimum in the direction of $X$ for a given $h_1$. 
However, one can still prove this by computing the partial derivative of $\Pi$ with respect to $X$, since the explicit formula of $\Pi$ can be determined (see~\ref{app:computemubar}). 
Therefore, one can keep the same definition of the sequences~\eqref{eq:optseq}.
Lemma~\ref{lem:limprod} remains true in this case by the definition of the sequence $(X_n, h_n)$.
However the last limit in Theorem~\ref{thm:lim} will be different as shown in the next theorem.

\begin{theorem}\label{thm:limbis}
$\lim_{n\rightarrow \infty} X_n = \infty$, $\lim_{n\rightarrow \infty} h_n = 0$ and $\lim_{n\rightarrow \infty} \Pi(X_n, h_n) = +\infty$.
\end{theorem}

\begin{proof}
Following the similar strategy as in Theorem~\ref{thm:lim}, one can show that $\lim_{n\rightarrow \infty} X_n = \infty$, $\lim_{n\rightarrow \infty} h_n = 0$.
On the other hand, by the construction of the two sequences $(X_{n-1} ,h_n)_{n>0}$, $(X_n ,h_n)_{n>0}$, one has $\Pi(X_{n-1}, h_n)< \Pi(X_n, h_n) < \Pi(X_n, h_{n+1}).$
Using~\eqref{eq:PtoPi} and passing $n$ to the limit, one finds 
\begin{equation*}
\begin{split}
&\lim_{n\rightarrow \infty} \Pi(X_{n-1}, h_n) = \lim_{n\rightarrow \infty} \frac{X_{n-1}}{\varepsilon(X_{n-1})}P(Y_{\text{opt}}) = \frac{P(Y_{\text{opt}})}{\alpha_0}  \lim_{n\rightarrow \infty} X_{n-1}^{1-s} = +\infty,\\
&\lim_{n\rightarrow \infty} \Pi(X_n, h_{n+1}) = \lim_{n\rightarrow \infty} \frac{X_n}{\varepsilon(X_n)}P(Y_{\text{opt}}) = \frac{P(Y_{\text{opt}})}{\alpha_0}  \lim_{n\rightarrow \infty} X_n^{1-s} = +\infty.
\end{split}
\end{equation*}
Therefore, $\lim_{n\rightarrow \infty} \Pi(X_{n-1}, h_n) = +\infty$.
\end{proof}

Note that for real  reactors, there is  a constraint on the minimal reactor depth $h_{\lim}$ (below which mixing is no more possible). 
An optimal solution can then be found in this case.
Indeed, as shown in Theorem~\ref{thm:limbis} (or Theorem~\ref{thm:lim}), a higher productivity can be obtained for higher biomass concentration and smaller reactor depth. 
Considering the minimal reactor depth, one can find the optimal biomass concentration  maximizing the productivity.




\section{Optimal control implementation in closed loop}\label{sec:controller}

As shown in previous section, there exists optimal biomass concentration for a given reactor depth $h$.
In this section, let us show that the evolution of the biomass concentration $X$ (defined in~\eqref{eq:dotX}) can be stabilized to a desired biomass concentration by applying an appropriate controller. 
More precisely, we consider the dilution rate $D$ in~\eqref{eq:dotX} as a controller.
Let us denote by $X^\star\in(0, X(0))$ the desired biomass concentration.
\begin{assumption}[H1]
\renewcommand{\labelenumi}{\alph{enumi}.} 
We assume that:
\begin{enumerate}
\item the quantity $\Phi:=(\bar \mu(X,h) -R)X$ is measured on-line from the plant,
\item the growth rate for the influent light intensity is larger than the respiration (i.e. $\mu(I_s)>R$),
\item the maximal dilution rate $D_{\max}$ is larger than the maximal growth rate $\mu_{\max}$.
\end{enumerate}
\end{assumption}

The quantity $\Phi$ denotes the average oxygen production which is available from the reactor.
Indeed, oxygen sensors or numerical estimators can be applied to obtain the quantity $\Phi$.
In the sequel, we assume that (H1) holds.
Then we have the following result.

\begin{proposition}\label{prop:controller}
The control law 
\begin{equation}\label{eq:D}
D =\left\{
\begin{array}{lr}
D_{\max} & X \geq \bar{X}\\
(\bar \mu(X,h) - R)\frac X{X^\star} & X < \bar{X}
\end{array}
\right.
\end{equation}
globally stabilizes equation~\eqref{eq:dotX} towards the positive point $X^\star$.
\end{proposition}

\begin{remark}
$\bar{X}>X^\star$ is chosen to determine the area where the control will be at its maximum rate. It is defined so that  $(\mu_{\max} - R)\frac{\bar{X}}{X^\star}<D_{max}$.

\end{remark}

\begin{proof}
From the definition of~\eqref{eq:D}, the control variable $D$ is positive. 
On the other hand, $\bar\mu(0,h)>R$, $\lim_{X\rightarrow \infty} \bar \mu(X,h)= 0$ and $\bar \mu(\cdot,h)$ is continuously decreasing with respect to $X$.
If the initial state $X(0)\geq \bar{X}$, then replacing $D = D_{\max}$ into~\eqref{eq:dotX} gives
\begin{equation*}\label{eq:Xlarge}
\dot X = (\bar \mu(X,h) - R - D_{\max})X.
\end{equation*}
In particular, $\bar \mu(X(0),h) - R - D_{\max}<0$, hence there exists a time $t_1>0$ such that the state $X$ decreases from $0$ to $t_1$ and $X(t_1)=\bar{X}$.
When $t>t_1$, $D = \frac{\Phi}{X^\star}$.
Replacing $D = \frac{\Phi}{X^\star}$ into~\eqref{eq:dotX} gives
\begin{equation}\label{eq:Xsmall}
\dot X = (\bar \mu(X,h) - R)\frac X{X^\star}(X^\star - X) = \frac{\Phi}{X^\star}(X^\star - X).
\end{equation}
Note that the system is now in the positive invariant region $X < \bar{X}$ and cannot come back again to $X \geq \bar{X}$. Moreover, $0<\bar \mu(X,h) - R<\mu_{\max} -R$. Then, integrating~\eqref{eq:Xsmall} gives $\forall t\geq t_1, \quad 0 < X^\star \leq X(t) \leq X(0)$.

In the case the initial state $X(0)\in(0,\bar{X})$, the control variable $D=\frac{\Phi}{X^\star}$ and the evolution equation~\eqref{eq:dotX} once again becomes~\eqref{eq:Xsmall}, hence we follow the small strategy as above.

Finally we find in both two cases that
\begin{equation*}
\forall t\geq 0, \quad 0 < X^\star \leq X(t) \leq X(0).
\end{equation*}
Therefore, the state $X^\star$ is globally exponentially stable for the evolution equation~\eqref{eq:dotX} by using the control law~\eqref{eq:D}. 
\end{proof}

\section{Numerical results}\label{sec:num}

In this section, we will illustrate some optimal conditions to maximize the algal productivity.
In this way, we first introduce an algorithm to compute the sequences defined in~\eqref{eq:optseq}.
We then give the parameters that we use for the numerical experiments and show some numerical results.

\subsection{Numerical algorithm}

In practice, one can use the next algorithm to compute for two sequences $(X_{n-1} ,h_n)_{n > 0}$ and $(X_n ,h_n)_{n > 0}$ defined by~\eqref{eq:optseq}.
\begin{algorithm}
\begin{algorithmic}[1]
\State \textbf{Input}: $Y_{\text{opt}}$, $n_{\max}$ and $X_0$.
\State \textbf{Output}: $(X_n ,h_n)_{n > 0}$
\State Set $n:= 0$.
\While{$n<n_{\max}$}
\State Set $n = n + 1$.
\State Compute $h_n=Y_{\text{opt}}/\varepsilon(X_{n-1})$.
\State Compute $X_n$ such that d$\Pi_X(X_n, h_n) = 0$.
\EndWhile
\end{algorithmic}
\caption{Search Optimum}
\label{alg:opt}
\end{algorithm}

\subsection{Parameter settings}

The Han model parameters are taken from~\cite{Grenier2020} and recalled in Table~\ref{tab:han}.
\begin{table}[htpb]
\caption{Parameter values for Han Model.}
\begin{center}
\begin{tabular}{|c|c|c|}
\hline
$k_r$ & 6.8 $10^{-3}$ & $\si{s^{-1}}$\\
\hline
$k_d$ & 2.99 $10^{-4}$  & -\\
\hline
$\tau$ & 0.25 & $\si{s}$\\
\hline
$\sigma_H$ & 0.047 & $\si{m^2.\mu mol^{-1}}$\\
\hline
$k_H$ & 8.7 $10^{-6}$ & -\\
\hline
$R$ & 1.389 $10^{-7}$ & $\si{s^{-1}}$\\
\hline
\end{tabular}
\end{center}
\label{tab:han}
\end{table}
Parameters $\mu_{\max}=\SI{1.64}{d^{-1}}$, $\theta=4.09\times 10^{-7}$ and $I^*=\SI{202.93}{\mu mol.m^{-2}.s^{-1}}$ are then derived from equation~\eqref{eq:hantohal}.
The considered surface light intensity is $I_s=\SI{2000}{\mu mol.m^{-2}.s^{-1}}$.
For $s=1$, we take from~\cite{Martinez201811} the specific light extinction coefficient for the species {\it Chlorella pyrenoidosa} $\alpha_0=\SI{0.2}{m^2.g^{-1}}$ and the background turbidity $\alpha_1=\SI{10}{m^{-1}}$.
Note that for the case where $s<1$, we compute the coefficient $\alpha_0(s)$ to find the one providing an extinction coefficient as close as possible to the reference linear case which is  generally the one measured:
\begin{equation}\label{eq:alpha0}
\alpha_0(s) := \rm{arg min}_{X\in[X_{\min},X_{\max}]}| \alpha_0(1) X - \alpha_0(s) X^s |.
\end{equation}

\subsection{Numerical study}

In this section, we provide some numerical tests to illustrate the influence of the water depth $h$, the biomass concentration $X$ and the light extinction function $\varepsilon(X)$ on algal productivity.

\subsubsection{Evaluation of different light extinction coefficient}

As mentioned in the previous section, the light extinction coefficient $\alpha_0$ needs to be better estimated when $s<1$, in comparison with the reference case $s=1$.
For this reason, for a range of biomass concentration $X$ in $[0,1000]$ ($\si{g\cdot m^{-3}}$), we use~\eqref{eq:alpha0} to find $\alpha_0$ that provide the same average extinction rate.
Figure~\ref{fig:epsX} shows $\varepsilon(X)$ defined by~\eqref{eq:eps} for different values of $s$ when the background turbidity $\alpha_1>0$.
\begin{figure}[htpb]
\centering
\includegraphics[scale=0.3]{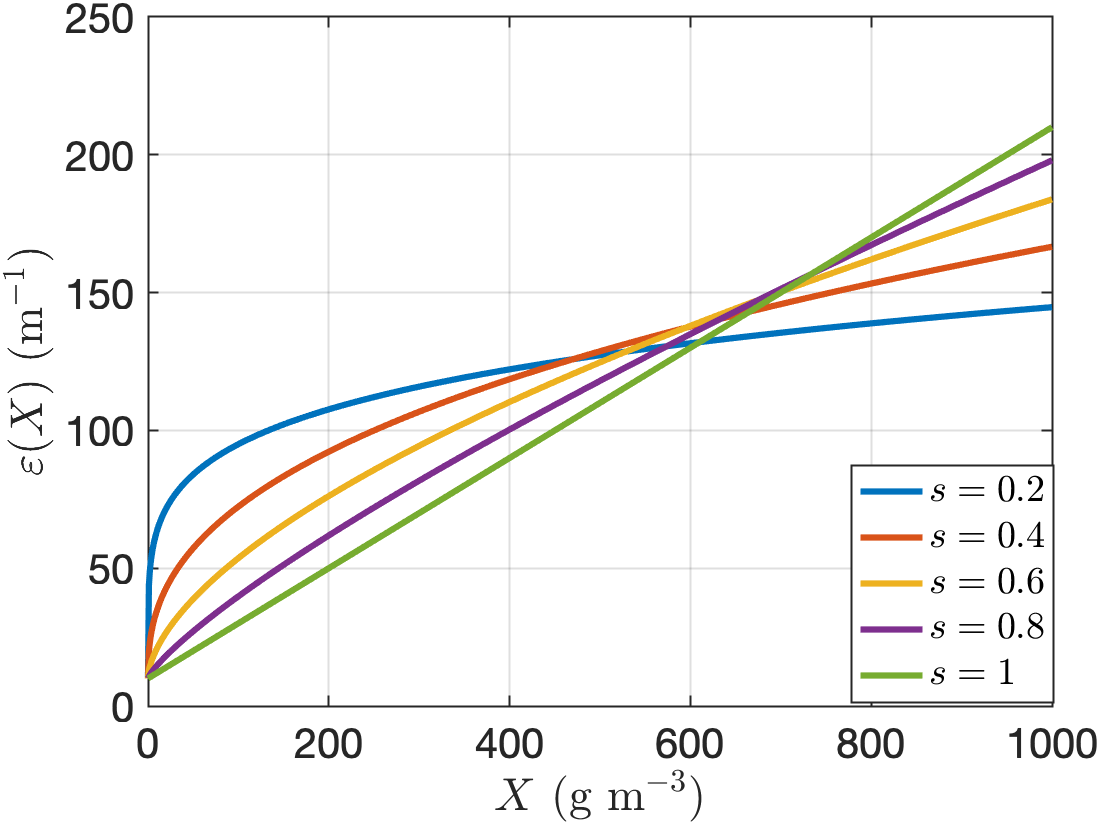}
\caption{$\varepsilon(X)$ with respect to $X$ for $s\in\{0.2,0.4,0.6,0.8,1\}$.}
\label{fig:epsX}
\end{figure}

\subsubsection{Global optimum of optical depth}

The optimal optical depth $Y_{\text{opt}}$ can be computed explicitly using~\eqref{eq:Yopt} once the  light intensity at the reactor surface  $I_s$ and the model parameters ($\theta$, $\mu_{\max}$, $I^*$, $R$)  are fixed.
Figure~\ref{fig:Pmuy} presents the evolution of the growth rate $\mu$ and optical depth productivity $P$ with respect to $y$ for different value of $s$ and $\alpha_1$.
\begin{figure}[htpb]
\centering
\includegraphics[scale=0.25]{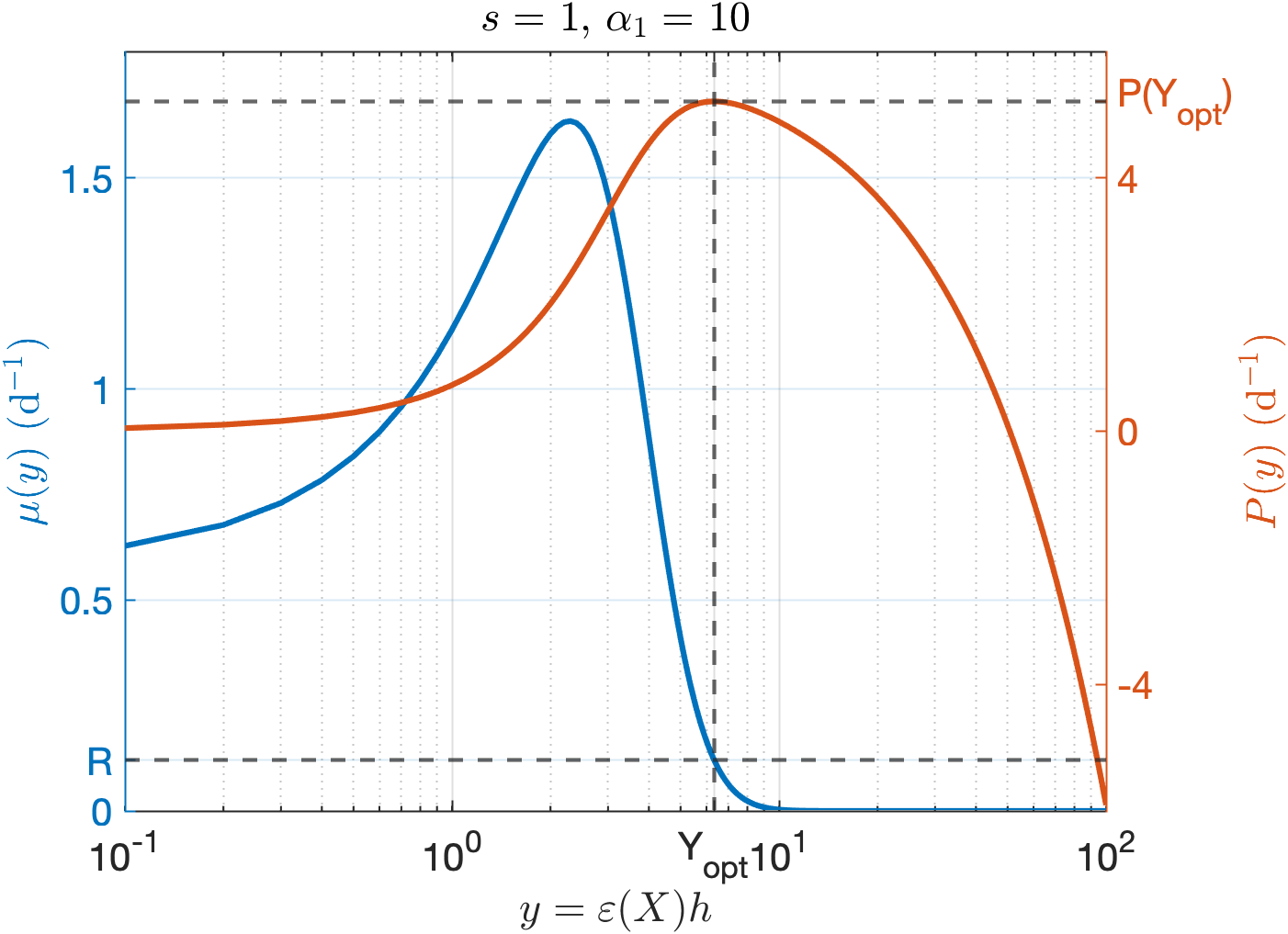}
\includegraphics[scale=0.25]{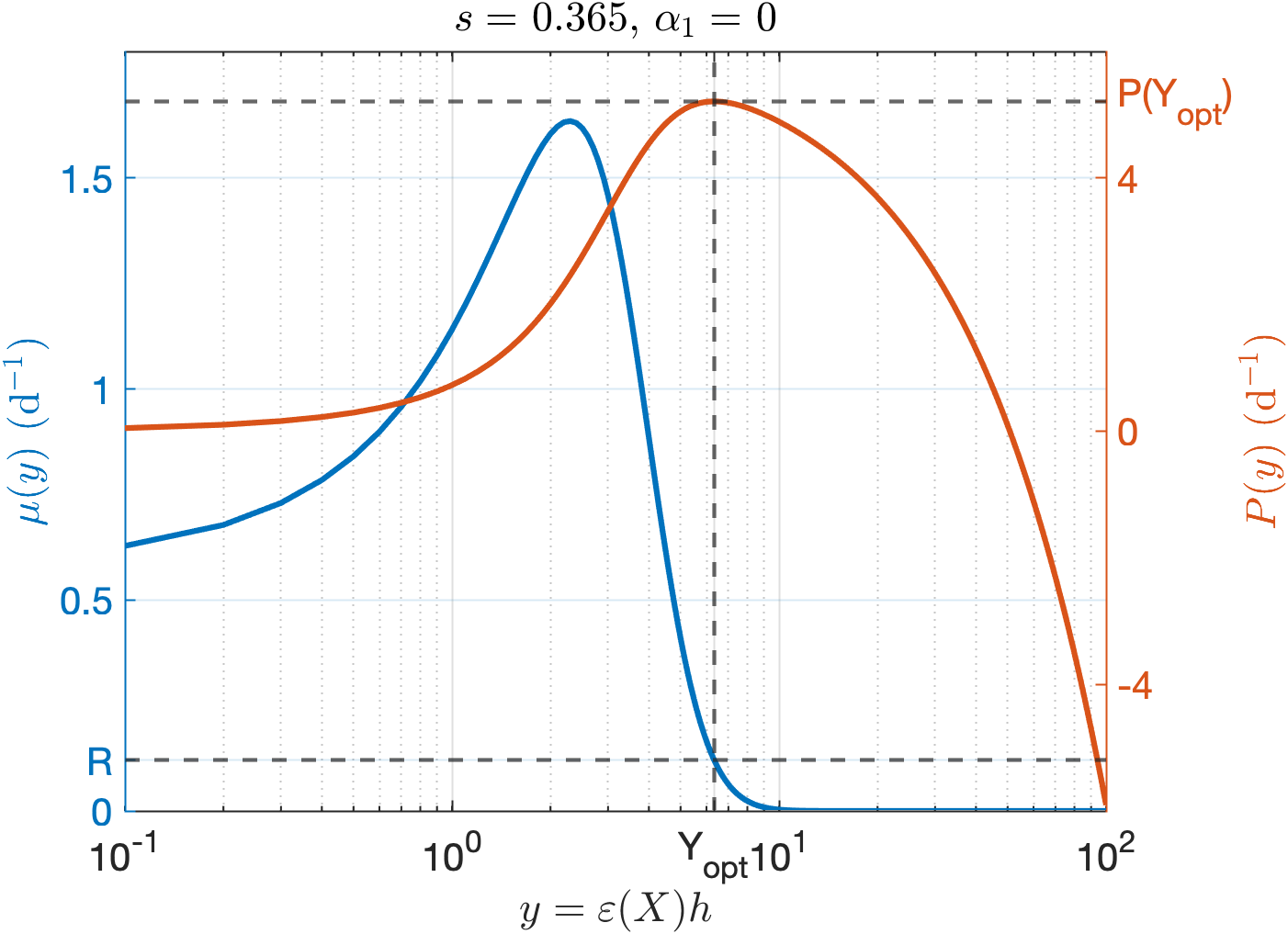}
\caption{Growth rate $\mu$ and optical depth productivity $P$ with respect to $y$.
Left: $s = 1$ and $\alpha_1=\SI{10}{m^{-1}}$.
Right: $s = 0.365$ and $\alpha_1=\SI{0}{m^{-1}}$.}
\label{fig:Pmuy}
\end{figure}
One can see that the optimum is obtained with $Y_{\text{opt}} = 6.337$, which also satisfies numerically $\mu(I(Y_{\text{opt}})) = R$.
Moreover, as mentioned in Remark~\ref{rem:Yopt}, $Y_{\text{opt}}$ does not change for other values of $\alpha_1$ and $s$.

In the same way, for a given biomass concentration $X$, Corollary~\ref{cor:hopt} provides a condition to determine the optimal depth to maximize the surface biomass productivity $\Pi$. 
Figure~\ref{fig:Pih} illustrates this corollary with a biomass concentration $X = \SI{50}{g \cdot m^{-3}}$ for different values of $s$ and $\alpha_1$.
\begin{figure}[htpb]
\centering
\includegraphics[scale=0.3]{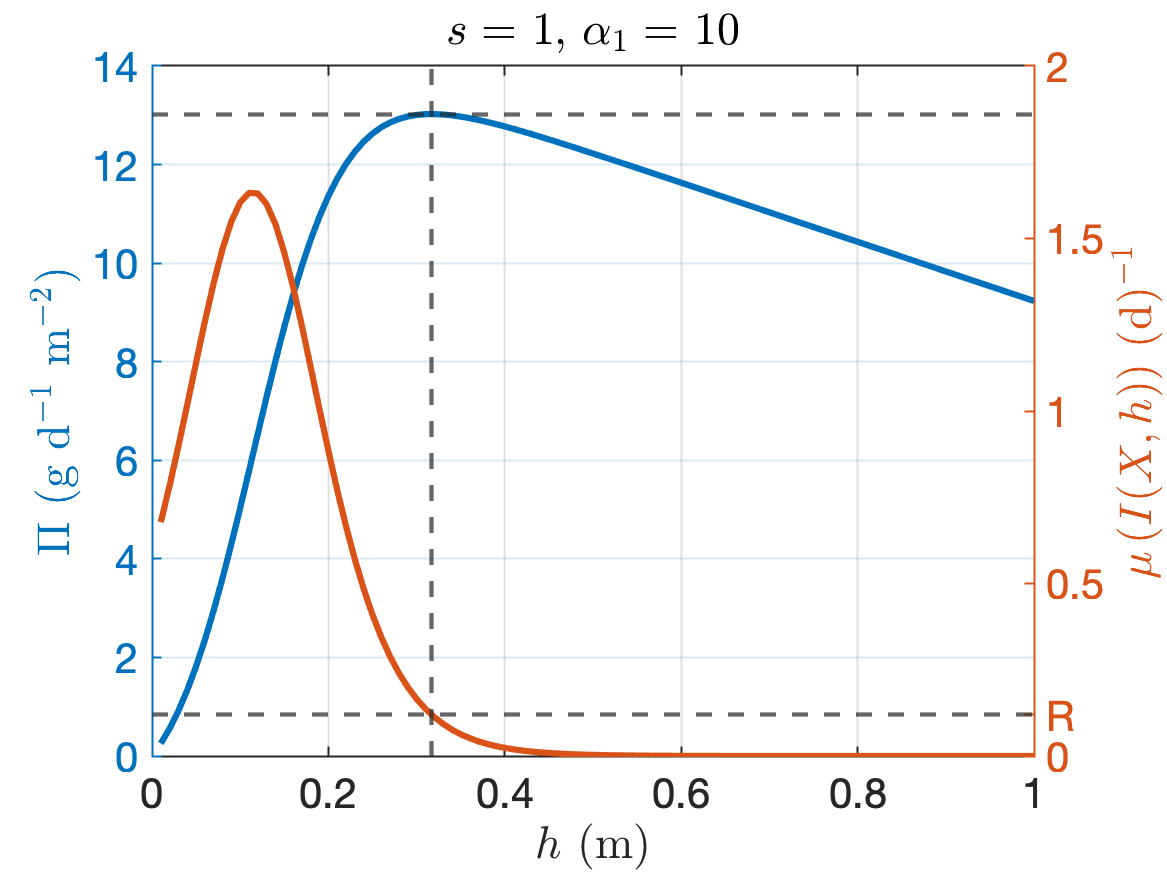}
\includegraphics[scale=0.3]{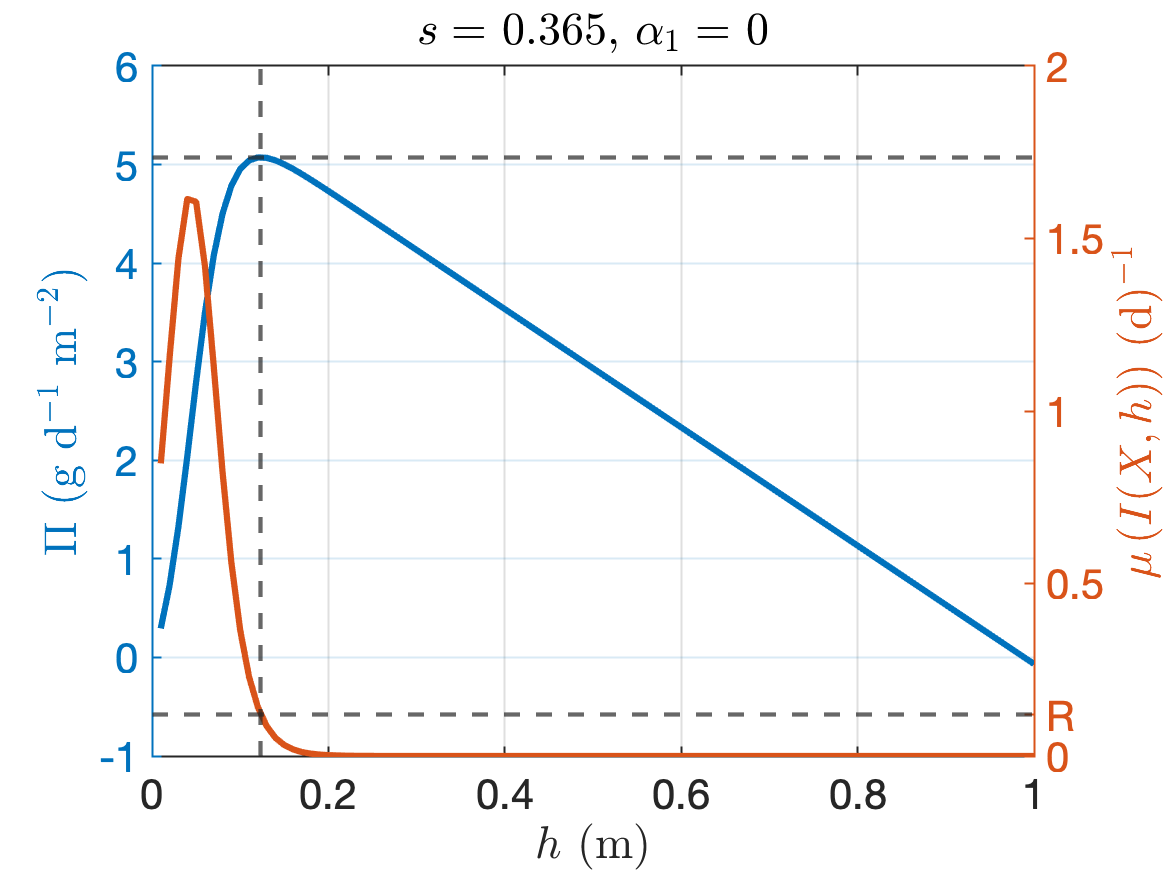}
\caption{Productivity ($\Pi$) and net growth rate ($\mu_{\text{net}}(X, h)$) with respect to depth ($h$) for $X = \SI{50}{g \cdot m^{-3}}$.
Left: $s = 1$ and $\alpha_1=\SI{10}{m^{-1}}$.
Right: $s = 0.365$ and $\alpha_1=\SI{0}{m^{-1}}$.}
\label{fig:Pih}
\end{figure}
Note that the optimal depth $h^*$ satisfies the relation $\varepsilon(X)h^* = Y_{\text{opt}}$. 
In other words, one can see that this optimum satisfies $\mu\left(I(X, h^*)\right) = R$.
It is worth remarking that the range of the productivity $\Pi$ changes for different value of $s$ and $\alpha_1$, this motivates the next test, where we study how these parameters affect algal growth.

\subsubsection{Influence of the background turbidity and \texorpdfstring{$s$}{Lg}}

Here we study the influence of the background turbidity $\alpha_1$ and the value of $s$ on the productivity $\Pi$. 
We keep the biomass concentration value $X=\SI{50}{g \cdot m^{-3}}$ and compute $h$ by using the relation $\varepsilon(X)h = Y_{\text{opt}}$ for different values of $\alpha_1$ and $s$.
Note that the depth $h$ computed in this way is the optimum to maximize the productivity for the given biomass concentration.
Figure~\ref{fig:Pialpha1} represents the optimal surface biomass productivity $\Pi$ with respect to the background turbidity.
\begin{figure}[htpb]
\centering
\includegraphics[scale=0.3]{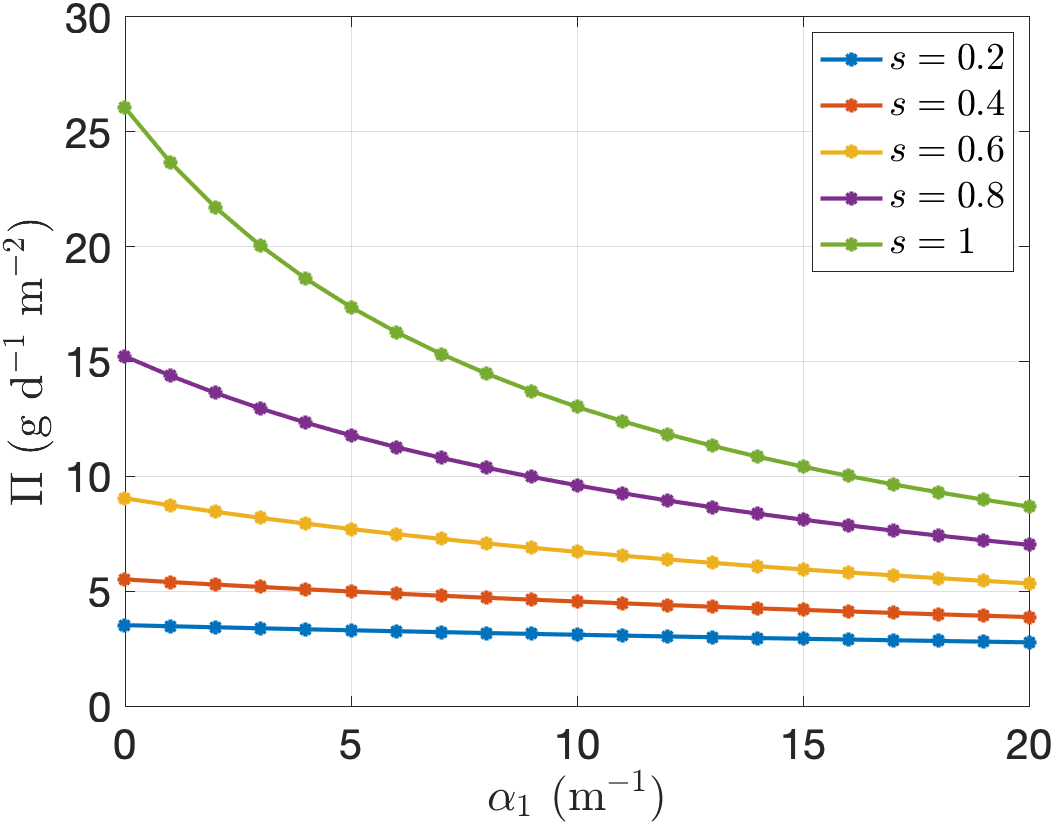}
\caption{Optimal surface biomass productivity with respect to the background turbidity $\alpha_1$ for $X=\SI{50}{g \cdot m^{-3}}$ and different value of $s$.}
\label{fig:Pialpha1}
\end{figure}
As we can expect intuitively, the larger the background turbidity is, the smaller the productivity is.
Furthermore, the productivity increases with the value of $s$ for a fixed value of turbidity $\alpha_1$.

\subsubsection{Local optimum in the case \texorpdfstring{$s=1$}{Lg}}

In reality, the depth $h$ depends on the type of reactors.
As an example, $h=\SI{0.1}{m}-\SI{0.5}{m}$ for raceway ponds, $h=\SI{1}{cm}-\SI{10}{cm}$ for tubular photobioreactors and $h=\SI{0.1}{mm}-\SI{1}{mm}$ for biofilm reactors (where the microalgal biomass is fixed on a support).
By knowing the lowest bound admissible for the reactor depth (depending on the process type), we only need to optimize the productivity in the direction of $X$.
Note that the turbidity $\alpha_1$ may change the optimal condition to maximize the surface biomass productivity $\Pi$.
Indeed, Figure~\ref{fig:PimuX} illustrates this for a reactor depth $h=\SI{0.15}{m}$.
\begin{figure}[htpb]
\centering
\includegraphics[scale=0.3]{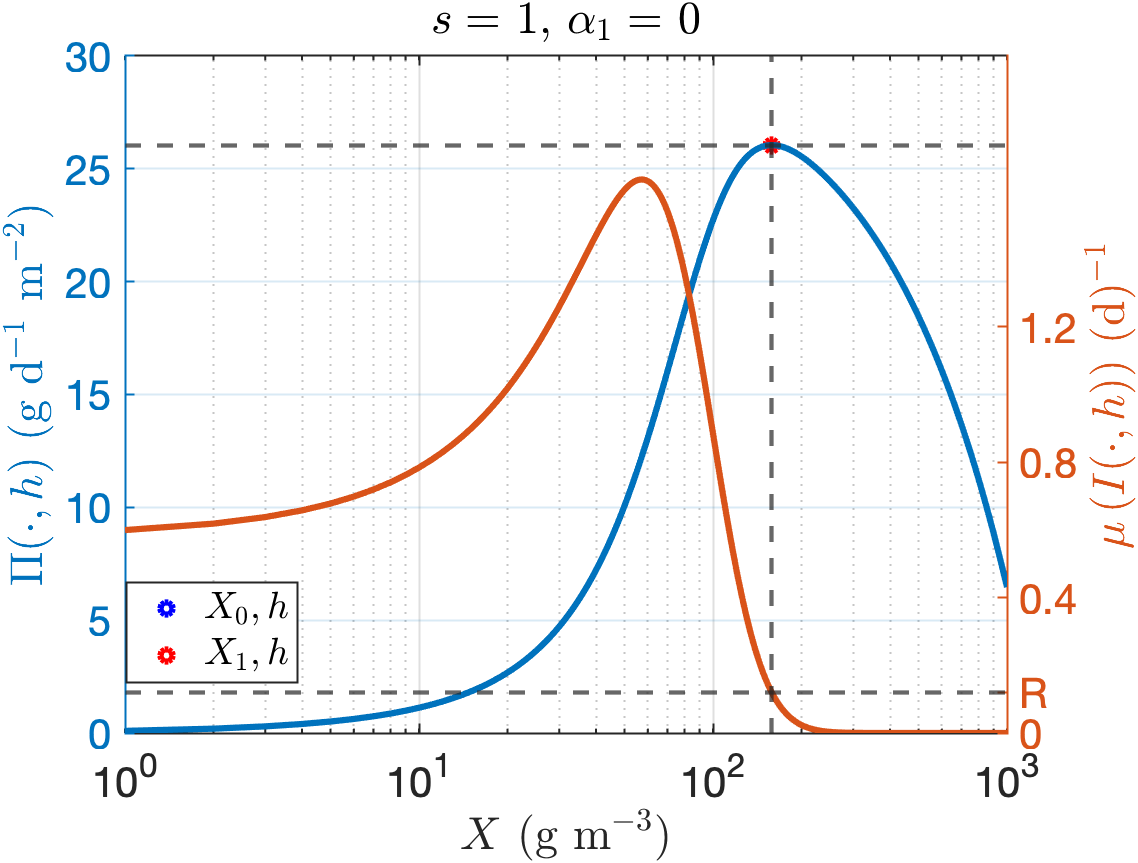}
\includegraphics[scale=0.3]{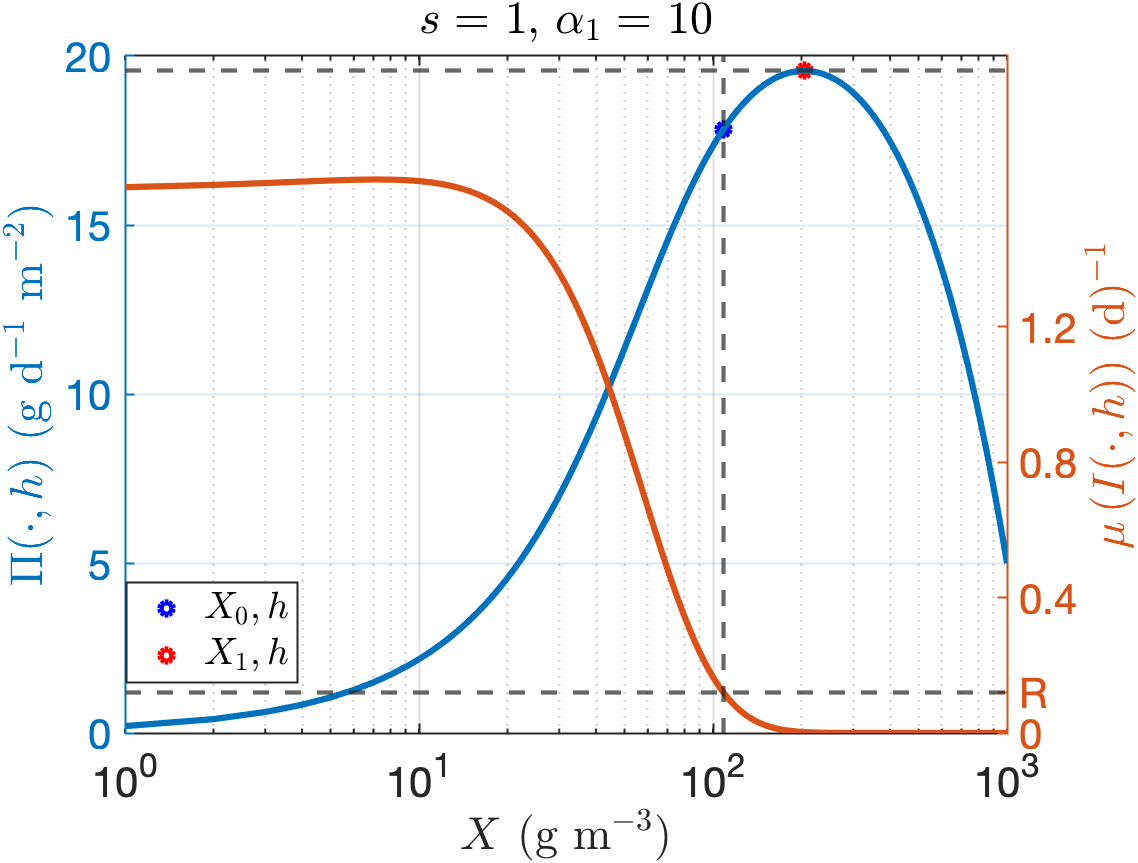}
\caption{Productivity ($\Pi$) with respect to biomass concentration ($X$) for $h = \SI{0.15}{m}$.
Left: $\alpha_1 = \SI{0}{m^{-1}}$.
Right: $\alpha_1=\SI{10}{m^{-1}}$.}
\label{fig:PimuX}
\end{figure}
Note that $X_0$ satisfies the relation $\varepsilon(X_0)h = Y_{\text{opt}}$ which also means that the net growth rate at the bottom of the reactor is zero (see the blue point in these two figures).
On the other hand, the red point $(X_1,h)$ is the optimum which maximize the surface biomass productivity $\Pi$ for this given depth $h$.
One can see that $X_0 = X_1=\SI{158.427}{g \cdot m^{-3}}$ in the case the background turbidity is zero in the system (Left), meaning that the optimum is the point which cancels the net average growth rate at the reactor bottom as we have mentioned in Corollary~\ref{cor:yoptnotur}.
However, Lemma~\ref{lem:dirX} indicates that by taking into account the background turbidity (Right), these two points are no longer the same and the optimum then satisfies $X_1= \SI{204.190}{g \cdot m^{-3}} > X_0=\SI{108.427}{g \cdot m^{-3}}$.

The global behaviour of the surface biomass productivity $\Pi$ is represented on Figure~\ref{fig:PiXh}, for $h\in(0, 1]$ and $X\in(0,1000]$.
To discuss the influence of the background turbidity, we consider two possible values, $\alpha_1=\SI{0}{m^{-1}}$ and $\alpha_1=\SI{10}{m^{-1}}$.
\begin{figure}[htpb]
\centering
\includegraphics[scale=0.3]{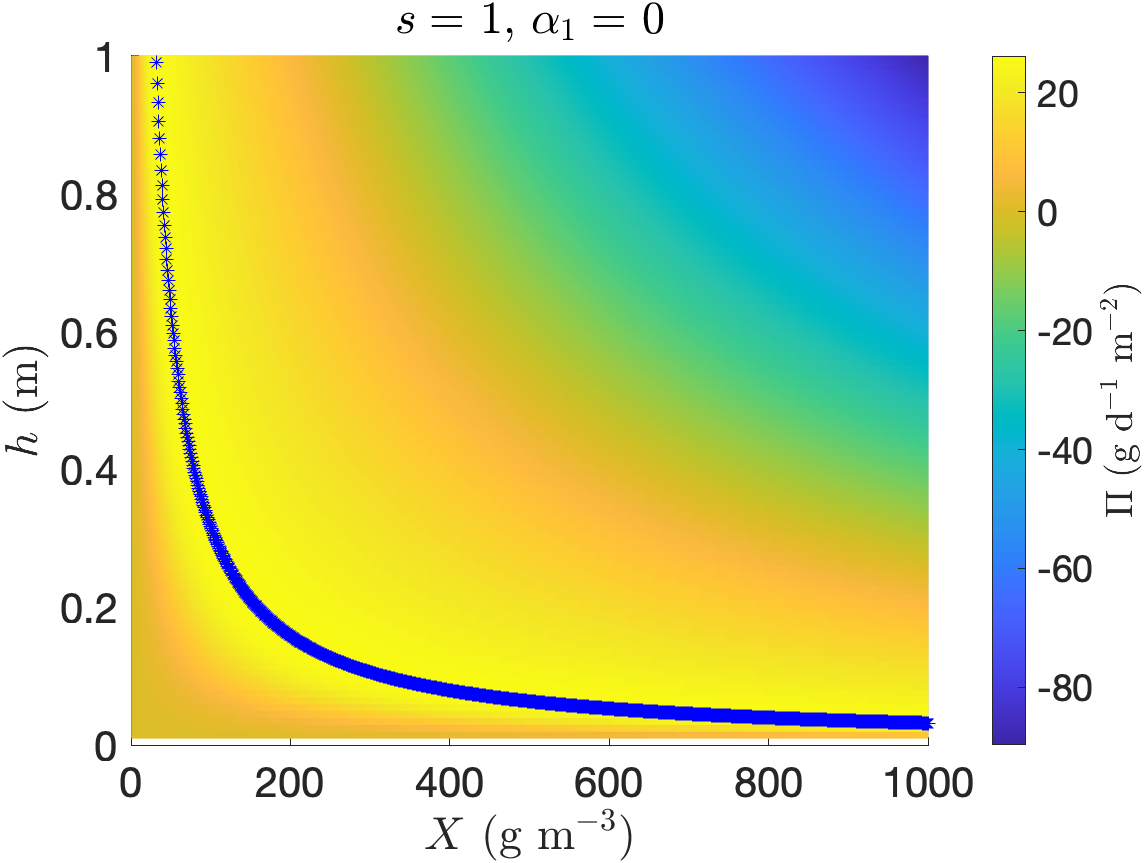}
\includegraphics[scale=0.3]{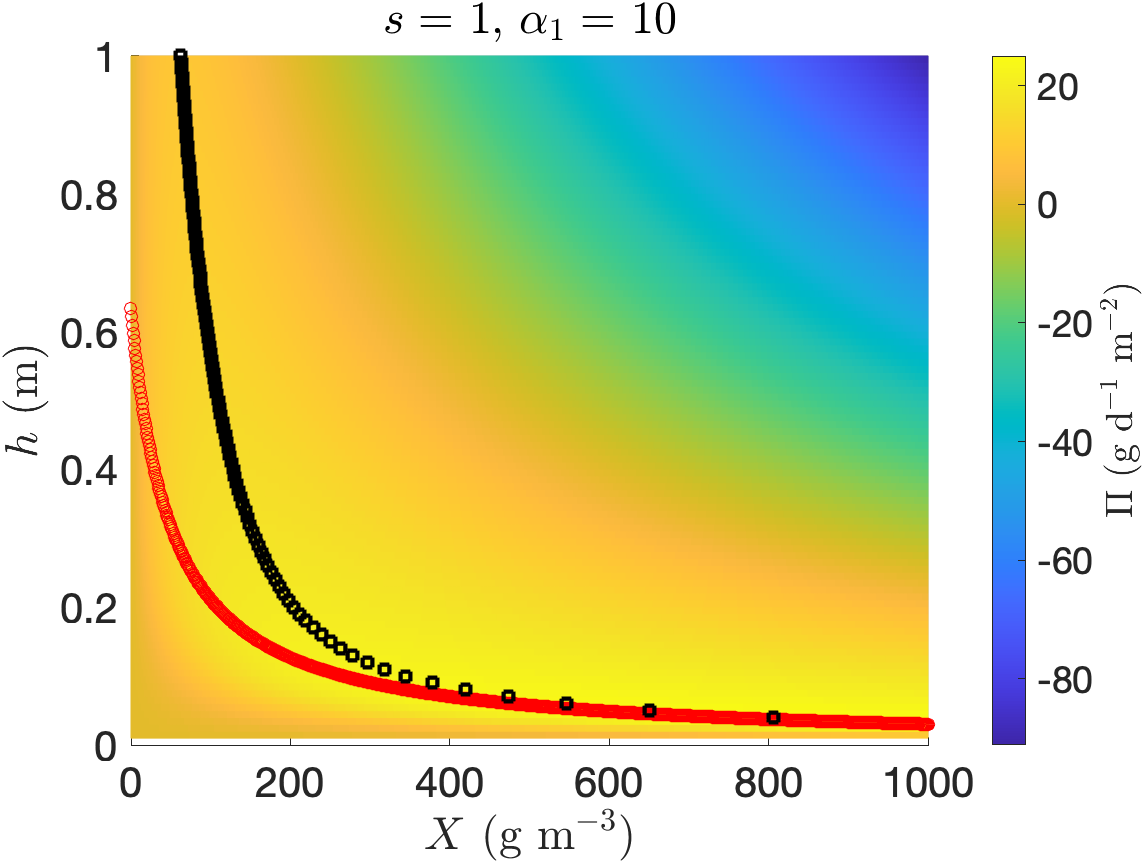}
\caption{Global behaviour of productivity ($\Pi$) with respect to depth ($h$) and biomass concentration ($X$).
Left: $\alpha_1 = \SI{0}{m^{-1}}$.
The blue stars represent the optimal couple $(X, h)$ such that $\Pi$ finds its global maximum.
Right: $\alpha_1=\SI{10}{m^{-1}}$.
The red circles represent the suboptimal couple $(X, h)$ where $\Pi$ finds its maximum in the direction of $h$ for a given $X$.
The black squares represent the suboptimal couple $(X, h)$ where $\Pi$ finds its maximum in the direction of $X$ for a given $h$.}
\label{fig:PiXh}
\end{figure}
Note that the blue points in the left figure $(X, h)$ satisfy the relation $\varepsilon(X)h = Y_{\text{opt}}$ which is also the global optimum.
However, by taking into account the background turbidity (see figure on the right), no global optimum exists as mentioned in Theorem~\ref{thm:globopt}.
Instead, for a given biomass concentration, the optimal depths can still be found using the relation $\varepsilon(X)h = Y_{\text{opt}}$ (represented by the red circles in the right figure).
For a given water depth, the optimal concentrations are obtained by cancelling the derivative of $\Pi(\cdot ,h)$ (represented by the black squares in the right figure).
Furthermore, one can observe that this two suboptima become closer when $X$ increases and $h$ decreases, meanwhile the productivity also increases in this direction. 

Let us set $X_0=\SI{50}{g \cdot m^{-3}}$, $\alpha_1=\SI{10}{m^{-1}}$ and $n_{\max}=10^4$. 
Figure~\ref{fig:Xnhn} illustrates the properties of these two sequences constructed by Algorithm~\ref{alg:opt}. 
\begin{figure}[htpb]
\centering
\includegraphics[scale=0.3]{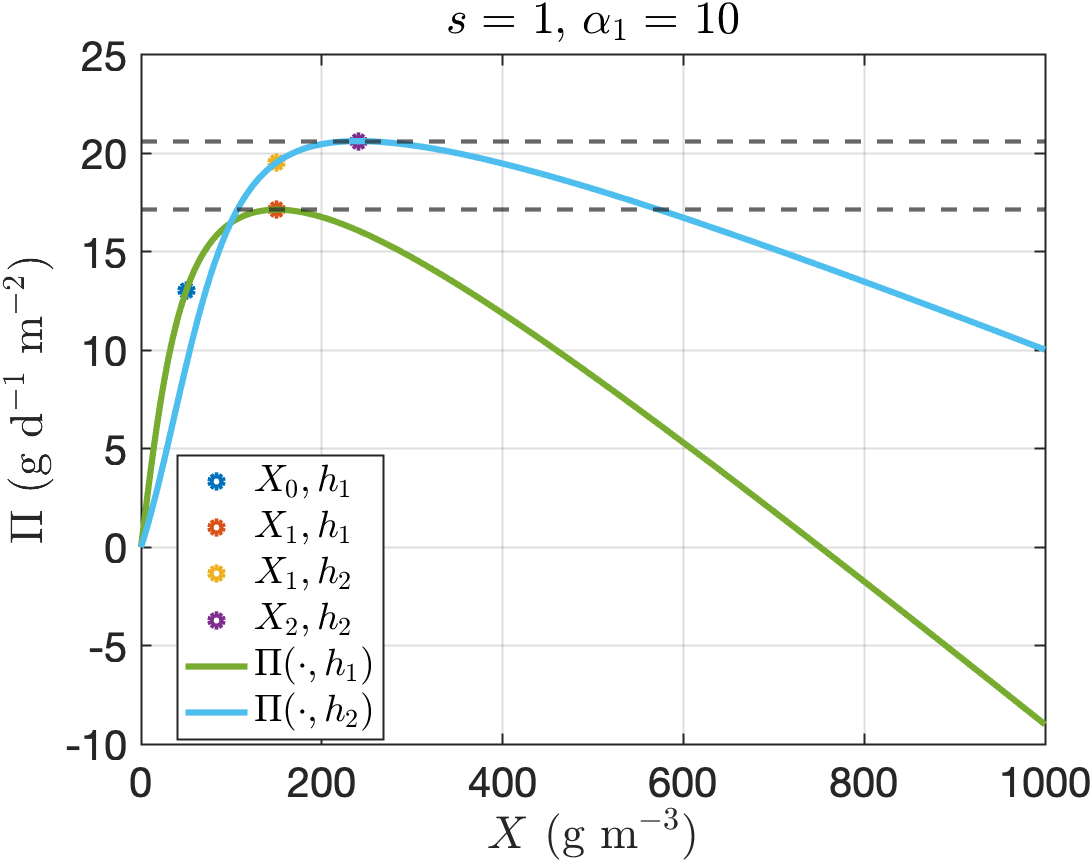}\\
\includegraphics[scale=0.28]{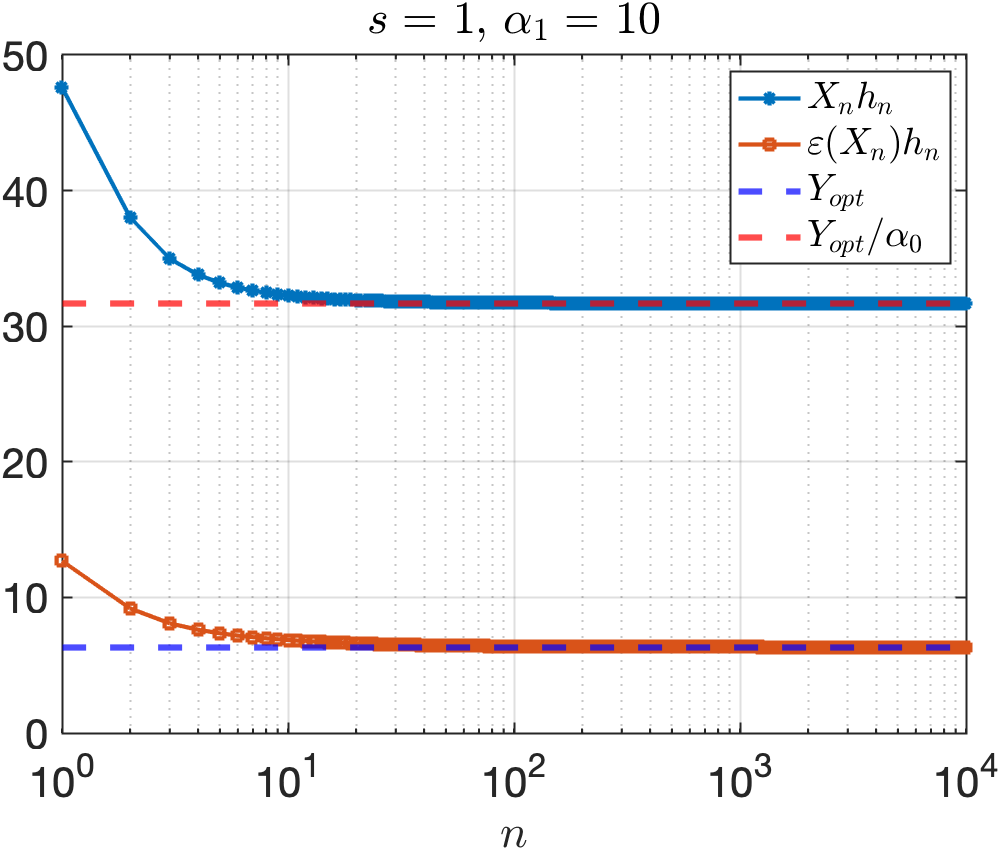}
\includegraphics[scale=0.28]{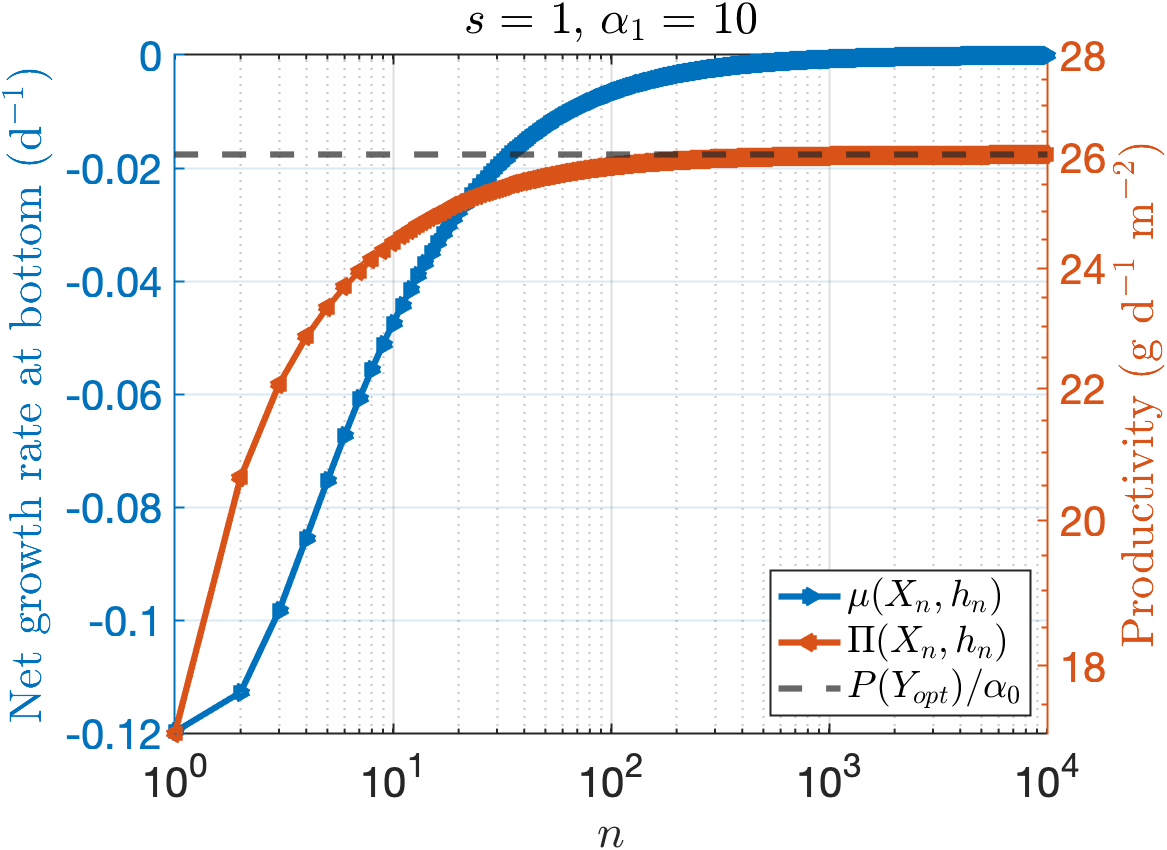}
\caption{Up: First-two elements of these two sequences.
Down Left: Surface biomass $X_n h_n$ and optical depth $\varepsilon(X_n)h_n$ for the sequence $(X_n ,h_n)_{n>0}$.
Down Right: Productivity $\Pi(X_n, h_n)$ and net growth rate at the reactor bottom $\mu(X_n, h_n) -R$  for the sequence $(X_n ,h_n)_{n>0}$.}
\label{fig:Xnhn}
\end{figure}
Starting from the figure on the top, the blue point and the yellow point are the first-two element of the sequence $(X_{n-1} ,h_n)_{n>0}$, the red point and the purple point are the first-two element of the sequence $(X_n ,h_n)_{n>0}$.
Recall that the sequence $(X_{n-1} ,h_n)_{n>0}$ always satisfies $\varepsilon(X_{n-1})h_n = Y_{\text{opt}}$ and the net growth rate at the reactor bottom is always 0.
We then only study the asymptotic behaviour of the sequence $(X_n, h_n)_{n>0}$.
As shown in bottom left figure, the surface biomass $X_n h_n$ converges to $\frac{Y_{\text{opt}}}{\alpha_0}$ and the optical depth $\varepsilon(X_n)h_n$ converges to $Y_{\text{opt}}$, as proved in Lemma~\ref{lem:limprod}.
The productivity $\Pi(X_n,h_n)$ converges to $P(Y_{\text{opt}})/\alpha_0$, see bottom right figure as proved in Theorem~\ref{thm:lim}.
Finally, the net growth rate at the reactor bottom converges to zero, which is the global optimum condition in the case where the background turbidity is 0 (see Corollary~\ref{cor:yoptnotur}).
In particular, since $(X_n, h_n)$ are the optima in the direction of $X$ for each $h_n$, one can see that the net growth rate at the reactor bottom for these optima are always negative, meaning that the compensation condition is only satisfied asymptotically.

\subsubsection{Local optimum in the case \texorpdfstring{$s<1$}{Lg}}

We start with a similar study as in Figure~\ref{fig:PimuX} in the case $s<1$.
Recall that the depth of the reactor is given by $h= \SI{0.15}{m}$ and two background turbidity values are given by $\alpha_1=\SI{0}{m^{-1}}$ and $\alpha_1=\SI{10}{m^{-1}}$.
Figure~\ref{fig:PimuXbis} illustrates the results for $s=0.365$. 
\begin{figure}[htpb]
\centering
\includegraphics[scale=0.3]{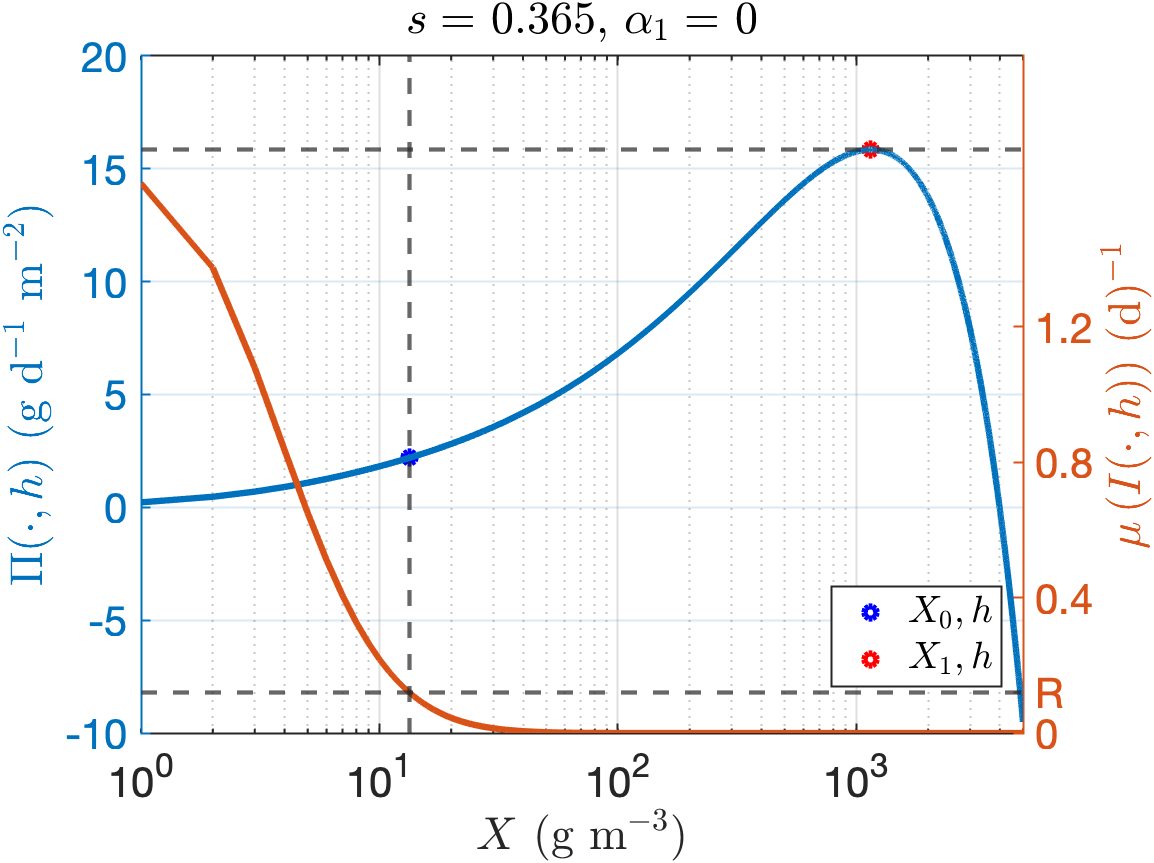}
\includegraphics[scale=0.3]{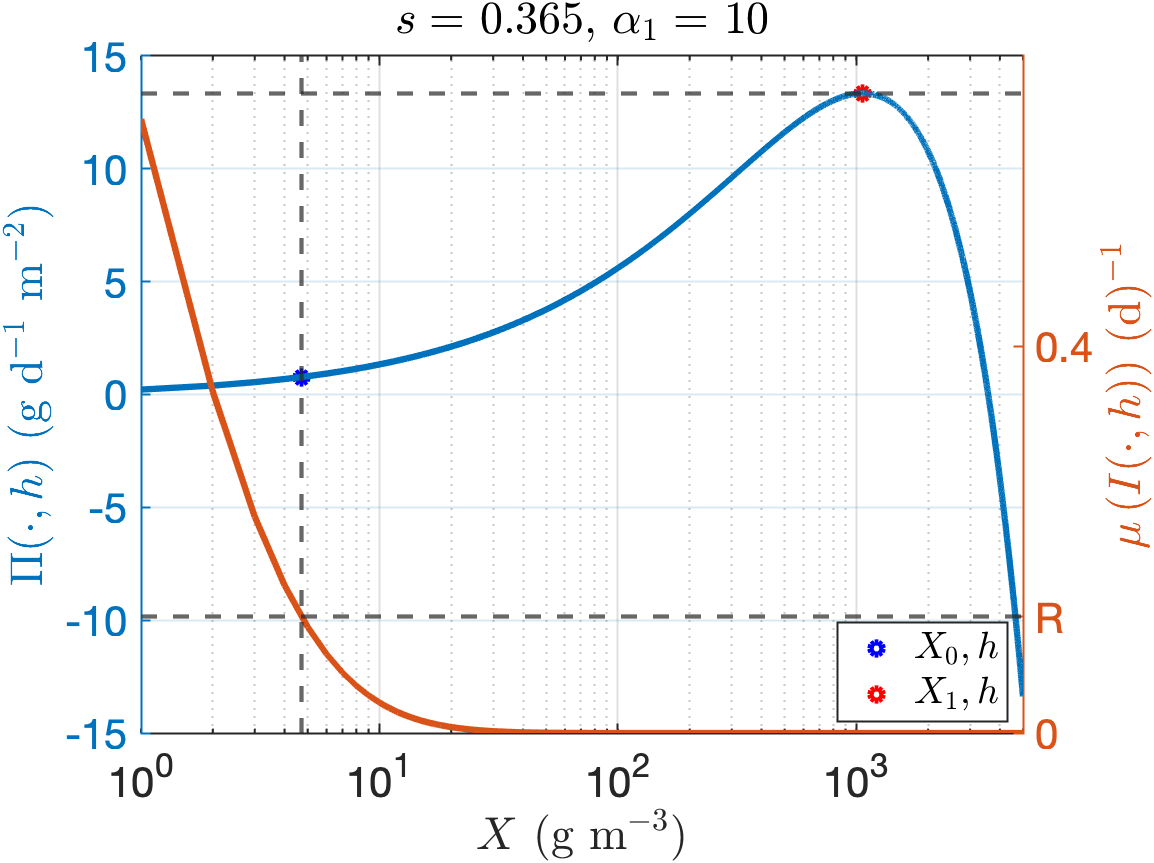}
\caption{Productivity ($\Pi$) with respect to biomass concentration ($X$) for $h = \SI{0.15}{m}$.
Left: $\alpha_1=\SI{0}{m^{-1}}$.
Right: $\alpha_1=\SI{10}{m^{-1}}$.}
\label{fig:PimuXbis}
\end{figure}
Recall that the blue point $(X_0,h)$ satisfies the relation $\varepsilon(X_0)h = Y_{\text{opt}}$ which also means that the net growth rate at the reactor bottom is zero, and the red point $(X_1,h)$ represents the optimum which maximizes the productivity for this depth $h$.
In the case $\alpha_1=\SI{0}{m^{-1}}$, we find $X_0=\SI{13.327}{g \cdot m^{-3}}$ and $X_1=\SI{1149.298}{g \cdot m^{-3}}$, whereas we obtain $X_0=\SI{4.715}{g \cdot m^{-3}}$ and $X_1=\SI{1064.574}{g \cdot m^{-3}}$ in the case $\alpha_1=\SI{10}{m^{-1}}$.
These two points do not coincide even when the background turbidity is zero, which is different from the case $s=1$.

Figure~\ref{fig:PiXhbis} presents the global behaviour of surface biomass productivity in the case $s=0.365$.
\begin{figure}[htpb]
\centering
\includegraphics[scale=0.3]{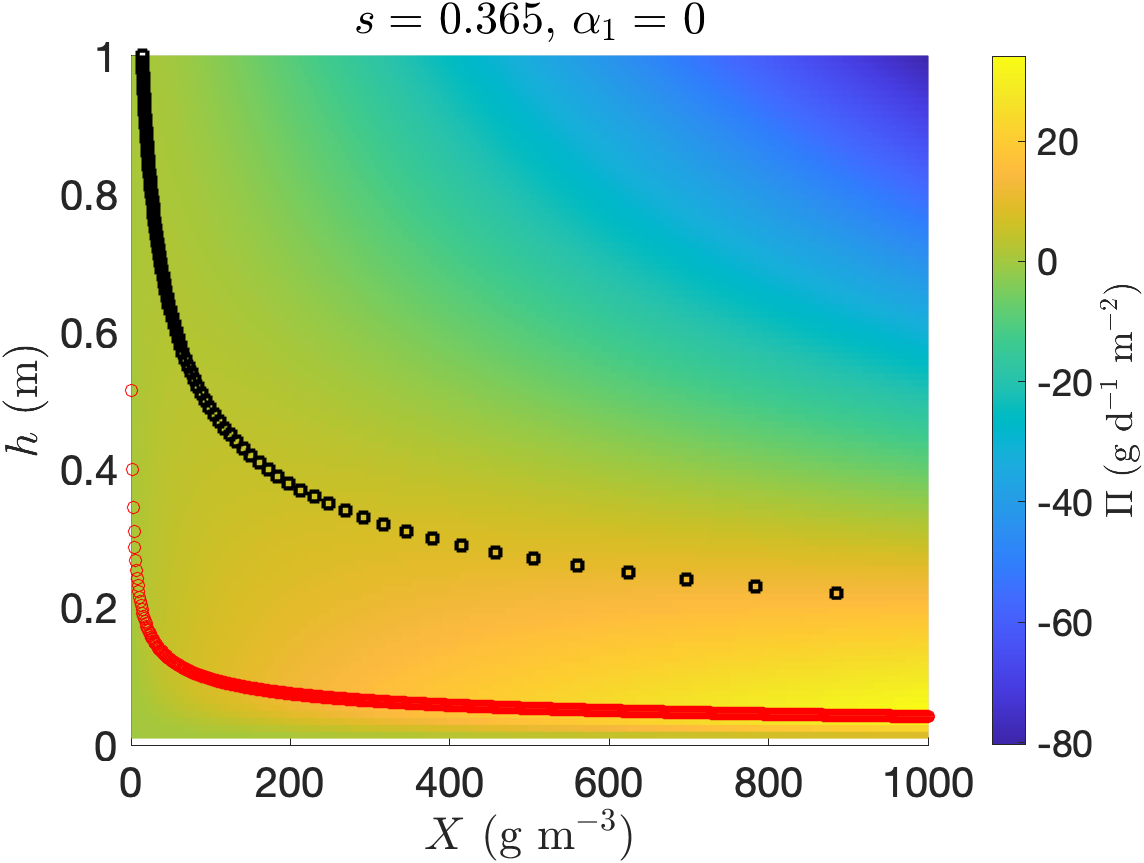}
\includegraphics[scale=0.3]{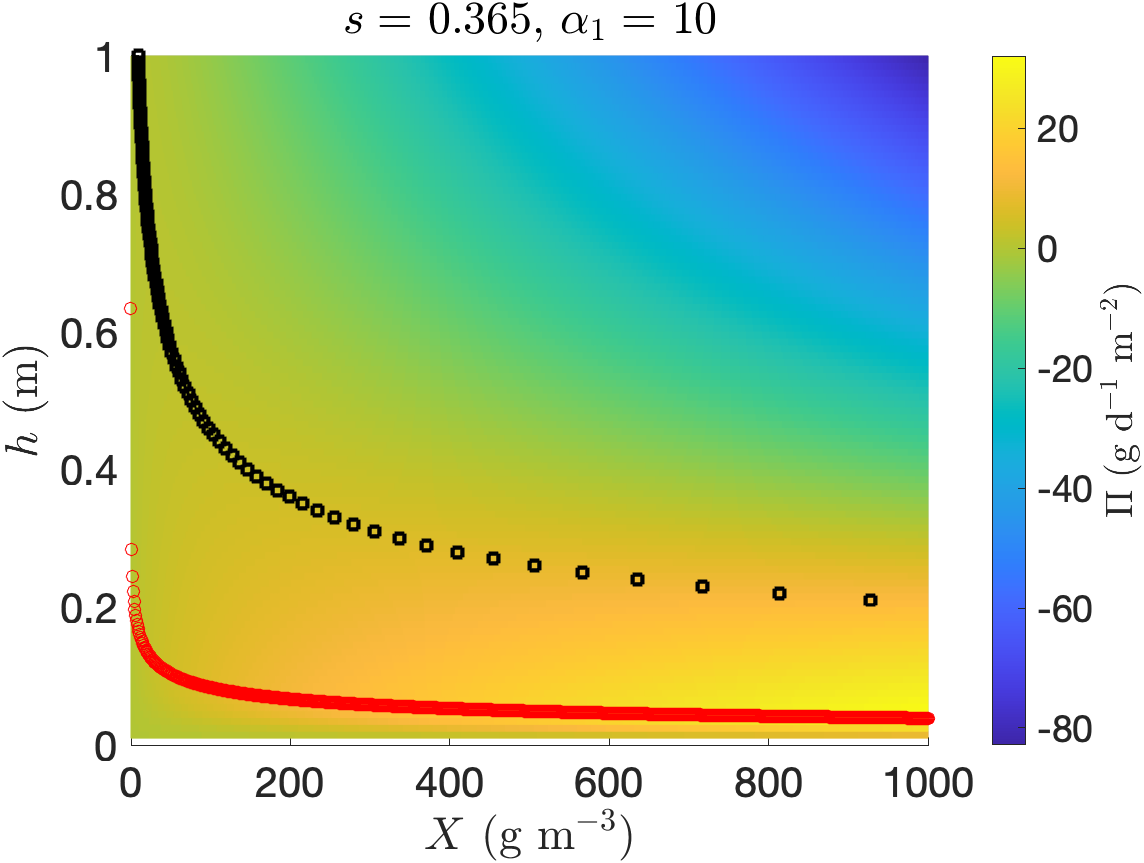}
\caption{Global behaviour of productivity ($\Pi$) with respect to depth ($h$) and biomass concentration ($X$).
Left: $\alpha_1 = \SI{0}{m^{-1}}$.
Right: $\alpha_1=\SI{10}{m^{-1}}$.
The red circles represent the suboptimal couple $(X, h)$ where $\Pi$ finds its maximum in the direction of $h$ for a given $X$.
The black squares represent the suboptimal couple $(X, h)$ where $\Pi$ finds its maximum in the direction of $X$ for a given $h$.}
\label{fig:PiXhbis}
\end{figure}
Unlike for the case $s=1$ (Figure~\ref{fig:PiXh}), the influence of the background turbidity becomes smaller when $s<1$.
However, similarly to this $s=1$ case (Right), the productivity becomes larger when the biomass concentration $X$ increases and the water depth $h$ decreases.
Furthermore, Figure~\ref{fig:Xnhnbis} shows the divergence of the productivity $\Pi$, as proved in Theorem~\ref{thm:limbis}.
\begin{figure}[htpb]
\centering
\includegraphics[scale=0.32]{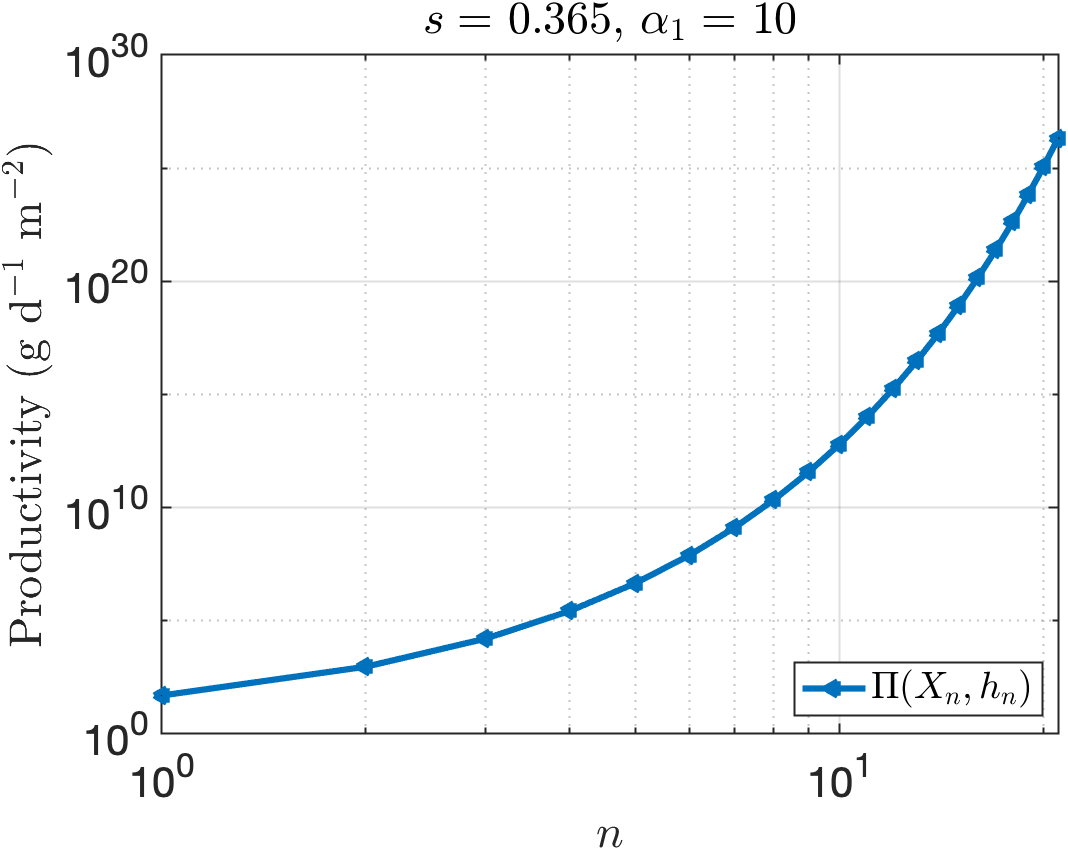}
\caption{$\Pi(X_n, h_n)$ for the sequence $(X_n ,h_n)_{n>0}$.}
\label{fig:Xnhnbis}
\end{figure}

\subsubsection{Controller test}

We present the efficiency of the controller $D$ designed in Proposition~\ref{prop:controller}.
Let us set $h=\SI{0.1}{m}$, $s=1$, $X(0)=[2500,50]\si{g.m^{-3}}$, $D_{\max}=10\mu_{\max}$ and keep other parameter settings.
Figure~\ref{fig:controller} illustrates the behaviour of the biomass concentration $X$ under our controller $D$.
\begin{figure}[htpb]
\centering
\includegraphics[scale=0.3]{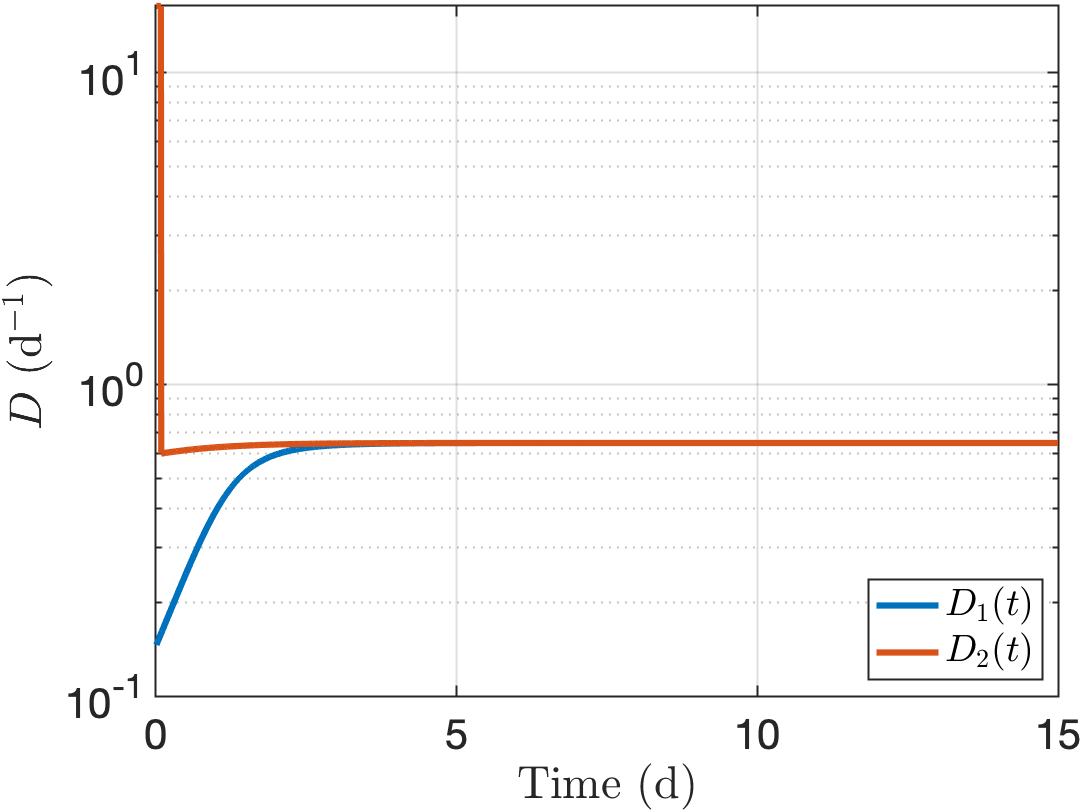}
\includegraphics[scale=0.3]{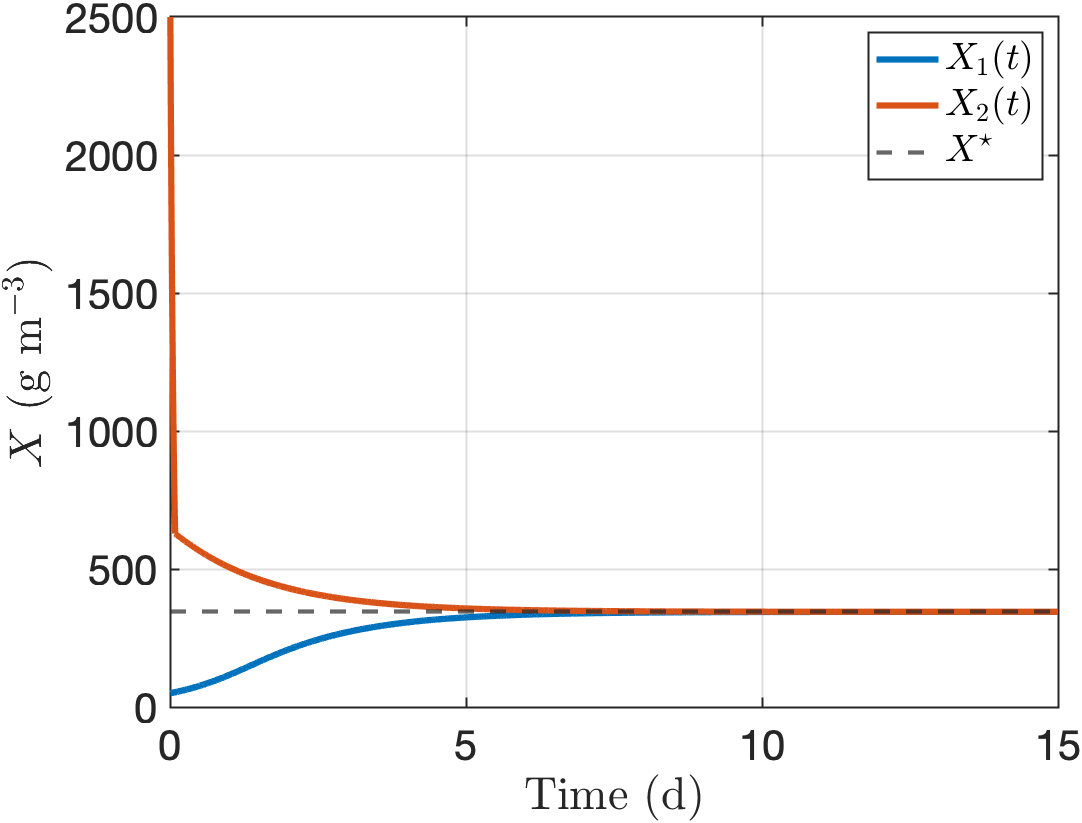}
\caption{Evolution of the biomass concentration $X$ in closed loop for two initial conditions.}
\label{fig:controller}
\end{figure}
Note that the desired biomass concentration $X^\star={\rm argmax}_{X\in\R_+} \Pi(X,h)$.
One can see that the evolution of the biomass concentration $X$ in closed loop converges to the desired optimal biomass concentration (after five days).

\section{Conclusion}

The concept of optical productivity $P$ has been defined and a global optimum $Y_{\text{opt}}$ has been found to maximize $P$.
This condition corresponds to a situation where the net growth rate at the reactor bottom is zero.
This optimum can be used to characterize the optimal water depth which maximizes surface biomass productivity $\Pi$ for a target biomass concentration value.
When the light extinction rate is affine with respect to the biomass concentration, an upper limit to the productivity is given which is obtained for an infinitely small depth and an infinitely large biomass concentration.

The proposed nonlinear controller stabilizes the biomass concentration to its optimal value $X^\star$. It could be improved by integrating an extremum seeking strategy~\cite{Liu2012,Guay2004} to automatically target the desired biomass without identifying it in advance.

\bibliographystyle{plain}
\bibliography{biblio}

\appendix

\section{Han model to Haldane description}\label{app:relation}

Let us consider the Han model~\cite{Han2002}:
\begin{equation}\label{eq:Han}
\left\{
\begin{array}{lr}
\dot{A} = -\sigma I A + \frac B{\tau},\\
\dot{B} =  \sigma I A - \frac B{\tau} + k_rC - k_d\sigma I B,\\
\dot{C} = -k_r C + k_d \sigma I B,
\end{array}
\right.
\end{equation}
where $A, B, C$ are the relative frequencies of the three possible photosynthetic states of the microalgae which satisfies 
\begin{equation}\label{eq:abc}
A+B+C=1. 
\end{equation}
Here $I$ is the photon flux density, a continuous time-varying signal. 
The other parameters are $\sigma$, that stands for the specific photon absorption, $\tau$ which is the turnover rate, $k_r$ which represents the photosystem repair rate and $k_d$ which is the damage rate. 
The system~\eqref{eq:Han} can be reduced to two equations by using~\eqref{eq:abc}.
Indeed, one can for instance eliminate $B$ in~\eqref{eq:Han} and gets
\begin{equation}\label{eq:Han_2eq}
\left\{
\begin{array}{lr}
\dot{A} = -(\sigma I + \frac 1{\tau}) A + \frac{1-C}{\tau},\\
\dot{C} = -(k_r+k_d\sigma I) C + k_d \sigma I (1-A),
\end{array}
\right.
\end{equation}
We then complete the system above with initial conditions
\begin{equation*}
\left(A(0), C(0)\right) = (A_0, C_0) \in \{(x, y) \in \R^2_+ | x + y \in [0, 1]\}.
\end{equation*}

The dynamics of the open state $A$ reaches its steady state following a process whose speed is very high compared to the dynamics of the photoinhibition state $C$~\cite{Hartmann2014} (for instance see Figure~\ref{fig:AC_evol}). 
\begin{figure}[htpb]
\centering
\includegraphics[scale=0.3]{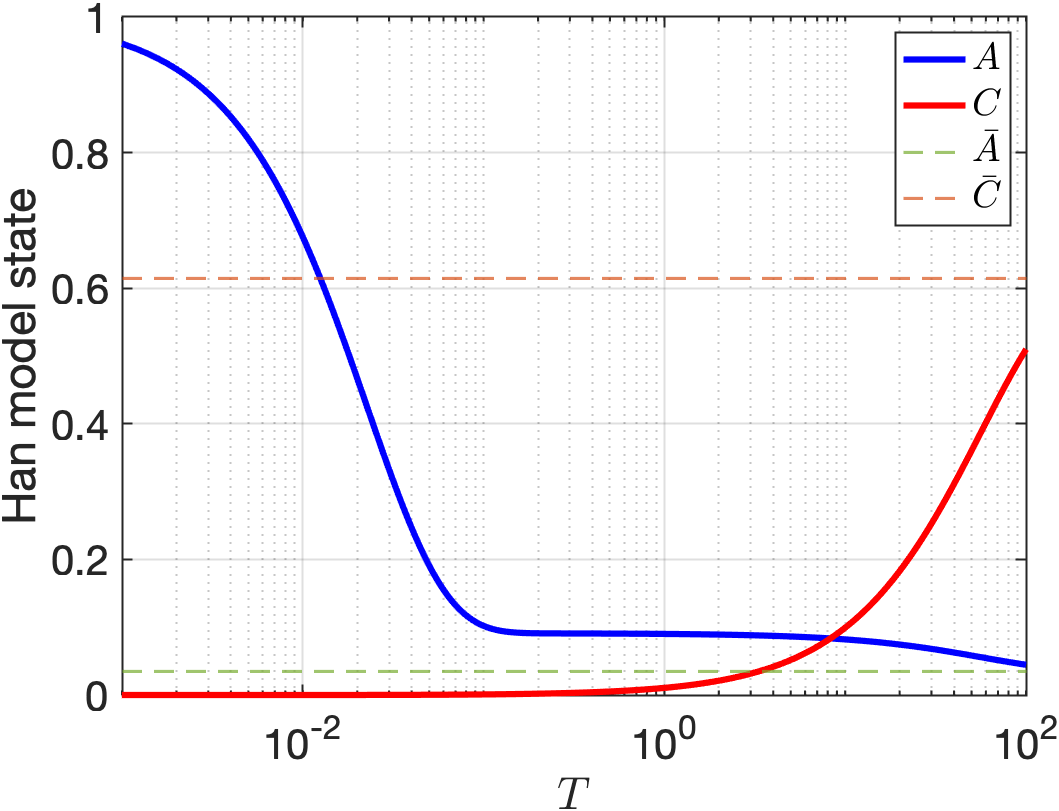}
\caption{Evolution for open state $A$ and photoinhibition state $C$ with the initial condition $(A(0),C(0)) = (1,0)$.
Note that we use an Euler Explicit scheme to solve the system~\eqref{eq:Han_2eq} by using the parameters presented in Table~\ref{tab:han}.
A $log$ scale is also used for the time variable.}
\label{fig:AC_evol}
\end{figure}
This phenomenon is mainly due to the presence of the multiplicative parameter $k_d$ which is relatively small (see Table~\ref{tab:han}, where an example of possible values for the Han parameters is given).
Since we usually focus on light variation at large time scale (larger than second) in real life applications, we can then apply a slow-fast dynamics using singular perturbation theory~\cite{Khalil2002}. 
More precisely, this consists in equating the first equation of~\eqref{eq:Han_2eq} to zero and find the pseudo steady state of $A$ as $\frac{1-C}{\tau\sigma_H I+1}$.
Replacing this into the second equation of~\eqref{eq:Han_2eq}, the previous two equations can finally be reduced to one equation, namely 
\begin{equation*}
\dot C = -(k_d\tau \frac{(\sigma I)^2}{\tau \sigma I+1} + k_r) C + k_d\tau \frac{(\sigma I)^2}{\tau \sigma I+1}.
\end{equation*}
The growth rate in the steady state of this system is then defined by
\begin{equation}\label{eq:muhan}
\mu_{\text{Han}}(I) := \frac{k \sigma I}{\frac{k_d}{k_r} \tau (\sigma I)^2 + \tau \sigma I + 1}.
\end{equation}
Finally by identifying~\eqref{eq:haldane} and~\eqref{eq:muhan} correctly, one finds
\begin{equation}\label{eq:hantohal}
\theta = k \sigma, \quad I^* = \sqrt{\frac{k_r}{k_d\tau\sigma^2}}, \quad \mu_{\max} = \frac{k \sigma}{\tau \sigma + 2 \sqrt{\frac{k_d \tau \sigma^2}{k_r}}}.
\end{equation}

\section{Explicit computations for average growth rate}\label{app:computemubar}

In this section, we provide details about the computation on $\bar\mu$.
As mentioned in Appendix~\ref{app:relation}, the growth rate described by Han model can be easily written in the Haldane description.
We choose to present the corresponding computation hereafter with Han model parameters. 
From the definition of $\bar \mu$~\eqref{eq:mubar}, one has
\begin{equation*}
\bar \mu =  \frac 1h \int_{-h}^0 \mu_{\text{Han}}(I(X, z)) \D z = \frac {k_r k \sigma}{h \varepsilon(X)} \int_{I_b}^{I_s} \frac 1{k_d \tau (\sigma I)^2 + k_r(\tau \sigma I + 1)} \D I,
\end{equation*}
where $I_b = I_s \exp(-\varepsilon(X) h)$.
Define $a= k_d\tau \sigma^2, b=k_r\tau\sigma, c=k_r$. 
According to the sign of the discriminant $\Delta$ of equation $a I^2 + b I + c$, three cases must be considered.

\begin{itemize}
\item $\Delta > 0$ : Then there exists two reel roots denoted by $d_1=\frac{-b+\sqrt{b^2-4a c}}{2a}$ and $d_2=\frac{-b-\sqrt{b^2-4a c}}{2a}$.
Hence one has
\begin{equation*}
\bar \mu = \frac {k_r k \sigma}{a h\varepsilon(X)} \Big(e_1\ln \left|\frac{I_s-d_1}{I_b-d_1}\right| + e_2\ln \left|\frac{I_s-d_2}{I_b-d_2}\right|\Big).
\end{equation*}
with $e_1+e_2 = 0$ and $e_1d_2 + e_2d_1 = -1$, i.e. $e_1=\frac 1{d_1-d_2} = \frac a{\sqrt{b^2-4a c}}$ and $e_2=\frac 1{d_2-d_1} = -\frac a{\sqrt{b^2-4a c}}$. 
Using $e_2 = -e_1$, we find
\begin{equation*}
\bar \mu = \frac {k_r k \sigma e_1}{a h\varepsilon(X)} \ln \left|\frac{(I_s-d_1)(I_b-d_2)}{(I_b-d_1)(I_s-d_2)}\right|.
\end{equation*}

\item $\Delta = 0$ : Then there exists a unique root denoted by $d = -\frac{b}{2a}$.
And one has
\begin{equation*}
\bar \mu  = \frac{k_r k \sigma}{h\varepsilon(X)} \frac{I_s-I_b}{(I_s-d)(I_b-d)} 
\end{equation*}

\item $\Delta < 0$ : Then one has
\begin{equation*}
a I^2 + b I + c =\frac{4a c - b^2}{4a}\Big(\big(\frac{ I+\frac b{2a} }{\sqrt{\frac{4a c - b^2}{4a^2}}} \big)^2 +1\Big).
\end{equation*}
Applying a change of variable by setting $y=\frac{ I+\frac b{2a} }{\sqrt{\frac{4a c- b^2}{4a^2}}}$, one gets
\begin{equation*}
\begin{split}
\bar \mu =\frac {k_r k \sigma}{h\varepsilon(X)}\frac{2}{\sqrt{4a c -b^2}}\big(\arctan(\frac{2a I_s +b}{\sqrt{4a c -b^2}})-\arctan(\frac{2a I_b +b}{\sqrt{4a c -b^2}})\big) .
\end{split}
\end{equation*}
\end{itemize}

Since one has the explicit formulation of the average growth rate $\bar \mu$, the surface biomass concentration $\Pi$ can be then computed explicitly including its derivatives. 
This is at the basis of the determination of the optimal biomass concentration $X$ for a given depth $h$.

\end{document}